\documentclass[10pt,english,journal]{IEEEtran}
\usepackage[T1]{fontenc}
\usepackage[latin9]{inputenc}
\usepackage{units}
\usepackage{amsthm}
\usepackage{amsmath}
\usepackage{amssymb}
\usepackage{esint}

\makeatletter
\theoremstyle{plain}
\newtheorem{thm}{\protect\theoremname}
\theoremstyle{definition}
\newtheorem{defn}[thm]{\protect\definitionname}
\theoremstyle{plain}
\newtheorem{cor}[thm]{\protect\corollaryname}
\theoremstyle{remark}
\newtheorem{rem}[thm]{\protect\remarkname}
\theoremstyle{plain}
\newtheorem{lem}[thm]{\protect\lemmaname}
\theoremstyle{plain}
\newtheorem{prop}[thm]{\protect\propositionname}
\theoremstyle{plain}
\newtheorem{conjecture}[thm]{\protect\conjecturename}

\usepackage{amsfonts}
\usepackage{version}
\usepackage{pgf}
\usepackage{xcolor}
\usepackage{cite}
\ifx\pdfoutput\relax\let\pdfoutput=\undefined\fi
\newcount\msipdfoutput
\ifx\pdfoutput\undefined\else
\ifcase\pdfoutput\else
\ifx\paperwidth\undefined\else
\ifdim\paperheight=0pt\relax\else\pdfpageheight\paperheight\fi
\ifdim\paperwidth=0pt\relax\else\pdfpagewidth\paperwidth\fi
\fi\fi\fi
\makeatother

\usepackage{babel}
\providecommand{\conjecturename}{Conjecture}
\providecommand{\corollaryname}{Corollary}
\providecommand{\definitionname}{Definition}
\providecommand{\lemmaname}{Lemma}
\providecommand{\propositionname}{Proposition}
\providecommand{\remarkname}{Remark}
\providecommand{\theoremname}{Theorem}

\begin{document}

\title{Thinning and Information Projections}

\author{Peter Harremo{\"e}s, \IEEEmembership{Member,~IEEE,} Oliver Johnson,
and Ioannis Kontoyiannis, \IEEEmembership{Member,~IEEE,}%
\thanks{Manuscript received xxx, 2015; revised xxx. This work was supported
by the European Pascal Network of Excellence.%
}%
\thanks{P. Harremo{\"e}s is with Copenhagen Business College, Denmark. O. Johnson
is with Bristol University, United Kingdom. I. Kontoyiannis is with
Athens University of Econ. \& Business, Greece.%
}}
\maketitle
\begin{abstract}
In this paper we establish lower bounds on information divergence
of a distribution on the integers from a Poisson distribution. These
lower bounds are tight and in the cases where a rate of convergence
in the Law of Thin Numbers can be computed the rate is determined
by the lower bounds proved in this paper. General techniques for getting
lower bounds in terms of moments are developed. The results about
lower bound in the Law of Thin Numbers are used to derive similar
results for the Central Limit Theorem.\end{abstract}


\section{Introduction and preliminaries}

\PARstart{A}{pproximation} by a Poisson distribution is a well
studied subject and a careful presentation can be found in \cite{Barbour1992}.
Connections to information theory have also been established \cite{Harremoes2001c,Kontoyiannis2005}.
For most values of the parameters, the best bounds on total variation
between a binomial distribution and a Poisson distribution with the
same mean have been proved by ideas from information theory via Pinsker's
inequality \cite{Csiszar1967,Fedotov2003,Harremoes2004,Johnson2007b}.

The idea of thinning a random variable was studied in \cite{Harremoes2010c}
and used to formulate and prove a Law of Thin Numbers that is a way
of formulating the Law of Small Numbers (Poisson's Law) so that it
resembles formulation of the Central Limit Theorem for a sequence
of independent identically distributed random variables. Here these
ideas will be developed further. There are three main reasons for
developing these results. The first is to get a lower bound for the
rate of convergence in the Law of Thin Numbers. The second is to use
these to get new inequalities and asymptotic results for the central
limit theorem. The last reason is to develop the general understanding
and techniques related to information divergence and information projection.
We hope eventually to be able to tell which aspects of important theorems
for continuous variables like the Entropy Power Inequality that can
be derived from results for discrete variables and which aspect are
essentially related to continuous variables. The relevance for communication
will not be discussed here, see \cite{Harremoes2010c} for some related
results.
\begin{defn}
Let $P$ denote a distribution on $\mathbb{N}_{0}$. For $\alpha\in\left[0,1\right]$
the $\alpha$\emph{-thinning} of $P$ is the distribution $\alpha\circ P$
given by
\[
\alpha\circ P\left(k\right)=\sum_{\ell=k}^{\infty}P\left(\ell\right)\binom{\ell}{k}\alpha^{k}\left(1-\alpha\right)^{\ell-k}.
\]

\end{defn}
If $X_{1},X_{2},X_{3},...$ are independent identically distributed
Bernoulli random variables with success probability $\alpha$ and
$Y$ has distribution $P$ independent of $X_{1},X_{2},\cdots$ then
$\sum_{n=1}^{Y}X_{n}$ has distribution $\alpha\circ P.$ Obviously
the thinning of an independent sum of random variables is the convolution
of thinnings.

Thinning transforms any natural exponential family on $\mathbb{N}_{0}$
into a natural exponential family on $\mathbb{N}_{0}.$ In particular
the following classes of distributions are conserved under thinning:
binomial distributions, Poisson distributions, geometric distributions,
negative binomial distributions, inverse binomial distributions, and
generalized Poisson distributions. This can be verified by direct
calculations \cite{Harremoes2007b}, but it can also be proved using
the variance function. A distribution on $\mathbb{N}_{0}$ is said
to be ultra log-concave if its density with respect to a Poisson distribution
is discrete log-concave. Thinning also conserves the class of ultra
log-concave distributions \cite{Johnson2007a}.

The thinning operation allow us to state and prove the Law of Thin
Numbers in various versions.
\begin{thm}[\cite{Harremoes2007b,Yu2009,Harremoes2010c} ]
\label{thm:ltnweak} Let $P$ be a distribution on $\mathbb{N}_{0}$
with mean $\lambda.$ If $P^{\ast n}$ denote the $n$-fold convolution
of $P$ then\end{thm}
\begin{enumerate}
\item $\frac{1}{n}\circ P^{\ast n}$ converges point-wise to $\mathrm{{Po}}\left(\lambda\right)$
as $n\rightarrow\infty.$
\item If the divergence $D\left(\frac{1}{n}\circ P^{\ast n}\Vert\mathrm{{Po}}\left(\lambda\right)\right)$
is finite eventually then
\[
D\left(\frac{1}{n}\circ P^{\ast n}\Vert\mathrm{{Po}}\left(\lambda\right)\right)\rightarrow0,\;\;\;\ \text{as }n\rightarrow\infty,
\]
 and the sequence $D\left(\frac{1}{n}\circ P^{\ast n}\Vert\mathrm{{Po}}\left(\lambda\right)\right)$
is monotonically decreasing.
\item If $P$ is an ultra log-concave distribution on $\mathbb{N}_{0}$
then 
\[
H\left(\frac{1}{n}\circ P^{\ast n}\right)\rightarrow H\left(\mathrm{{Po}}\left(\lambda\right)\right),\;\;\;\ \text{as }n\rightarrow\infty.
\]
 Furthermore the sequence $H\left(\frac{1}{n}\circ P^{\ast n}\right)$
is monotonically increasing. 
\end{enumerate}
The focus of this paper is to develop techniques that allow us to
give lower bounds on the rate of convergence in the Law of Thin Numbers.
One of our main results Theorem \ref{nedregraense} has the following
weaker result as corollary.
\begin{thm}
\label{ThmNedreVar}Let $X$ denote a discrete random variable with
$\mathrm{{E}}\left[X\right]=\lambda.$ Then
\[
2\left(D\left(X\Vert\mathrm{{Po}}\left(\lambda\right)\right)\right)^{\nicefrac{1}{2}}\geq1-\frac{\mathrm{{Var}}\left(X\right)}{\lambda}.
\]

\end{thm}
If $X\sim\mathrm{Bi}\left(n,\nicefrac{\lambda}{n}\right)$, then $\mathrm{{Var}}\left(X\right)=np\left(1-p\right)=\lambda\left(1-\nicefrac{\lambda}{n}\right).$
Hence
\[
D\left(X\Vert\mathrm{{Po}}\left(\lambda\right)\right)\geq\frac{\lambda^{2}}{4n^{2}}.
\]
 For the sequence of binomial distributions $\mathrm{Bi}\left(n,\nicefrac{\lambda}{n}\right)$
the rate of convergence in information is given by
\begin{equation}
n^{2}D\left(\mathrm{Bi}\left(n,\nicefrac{\lambda}{n}\right)\Vert\mathrm{{Po}}\left(\lambda\right)\right)\rightarrow\frac{\lambda^{2}}{4}~\text{for }n\rightarrow\infty,\label{EqBin}
\end{equation}
 which was proved in \cite{Harremoes2004} and \cite{Kontoyiannis2005}. 
\begin{cor}
Let $\mathrm{{Po}}_{\beta}\left(\lambda\right)$ denote the minimum
information distribution from $\mathrm{Po}\left(\lambda\right)$ and
with the same mean and variance as $\mathrm{Bi}\left(n,\lambda/n\right).$
Then
\[
n^{2}D\left(\mathrm{Bi}\left(n,\nicefrac{\lambda}{n}\right)\Vert\mathrm{{Po}}_{\beta}\left(\lambda\right)\right)\rightarrow0\ \text{for }n\rightarrow\infty.
\]
\end{cor}
\begin{rem}
The distribution $\mathrm{{Po}}_{\beta}\left(\lambda\right)$ can
be identified with an element in an exponential family that will be
studied in more detail in Section \ref{SecLower}. \end{rem}
\begin{proof}
According to the Pythagorean Inequality for information divergence
\cite{Csiszar1975} we have 
\begin{align*}
D\left(\mathrm{Bi}\left(n,\nicefrac{\lambda}{n}\right)\Vert\mathrm{{Po}}\left(\lambda\right)\right) & \geq D\left(\mathrm{Bi}\left(n,\nicefrac{\lambda}{n}\right)\Vert\mathrm{{Po}}_{\beta}\left(\lambda\right)\right)\\
 & \qquad+D\left(\mathrm{{Po}}_{\beta}\left(\lambda\right)\Vert\mathrm{{Po}}\left(\lambda\right)\right)\\
 & \geq D\left(\mathrm{Bi}\left(n,\nicefrac{\lambda}{n}\right)\Vert\mathrm{{Po}}_{\beta}\left(\lambda\right)\right)\\
 & \qquad+\frac{\lambda^{2}}{4n^{2}}.
\end{align*}
 Multiplication by $n^{2}$ leads to
\begin{multline*}
n^{2}D\left(\mathrm{Bi}\left(n,\nicefrac{\lambda}{n}\right)\Vert\mathrm{{Po}}\left(\lambda\right)\right)\\
\geq n^{2}D\left(\mathrm{Bi}\left(n,\nicefrac{\lambda}{n}\right)\Vert\mathrm{{Po}}_{\beta}\left(\lambda\right)\right)+\frac{\lambda^{2}}{4}
\end{multline*}
 and
\begin{align*}
0 & \leq n^{2}D\left(\mathrm{Bi}\left(n,\nicefrac{\lambda}{n}\right)\Vert\mathrm{{Po}}_{\beta}\left(\lambda\right)\right)\\
 & \leq n^{2}D\left(\mathrm{Bi}\left(n,\nicefrac{\lambda}{n}\right)\Vert\mathrm{{Po}}\left(\lambda\right)\right)-\frac{\lambda^{2}}{4}.
\end{align*}
 The result follows by the use of (\ref{EqBin}).
\end{proof}
Since the second moment is the sufficient statistic in the exponential
family $\beta\rightarrow\mathrm{{Po}}_{\beta}\left(\lambda\right)$
and $\mathrm{Bi}\left(n,\nicefrac{\lambda}{n}\right)$ asymptotically
is very close to $\mathrm{{Po}}_{\beta}\left(\lambda\right)$ we essentially
prove that calculation of the second moment is asymptotically sufficient
for testing the binomial distribution versus the Poisson distribution. 

Pinsker's inequality can be used to give an upper bound on total variation
when the divergence is given. With a bound on total variation we get
a bound on any bounded function because
\[
\left\vert \int f\,\mathrm{d}P-\int f\,\mathrm{d}Q\right\vert \leq\sup\left\vert f\right\vert \cdot\left\Vert P-Q\right\Vert _{1}.
\]
 If $\int f\,\mathrm{d}Q=0$ we get
\[
-\sup\left\vert f\right\vert \cdot\left\Vert P-Q\right\Vert _{1}\leq\int f\,\mathrm{d}P\leq\sup\left\vert f\right\vert \cdot\left\Vert P-Q\right\Vert _{1}.
\]
 If $f$ is unbounded such bounds cannot be derived so we shall turn
our attention to functions that are lower bounded. For such functions
we have 
\[
\inf f\cdot\left\Vert P-Q\right\Vert _{1}\leq\int f\,\mathrm{d}P.
\]
 When $Q$ is a Poisson distribution and $f$ is a Poisson-Charlier
polynomial then a much more tight bound can be derived and in a number
of cases it will give the correct rate of convergence in the law of
small numbers and related convergence theorems.

The lower bounds that we derive can also be used to qualify the statement
\textquotedbl{}two moments suffice for Poisson approximation\textquotedbl{}
that has been the title of an article \cite{Arratia1989}. In the
article \cite{Arratia1989} used the Chen-Stein method to get bounds
on the rate of convergence in Poisson convergence in the total variation
metric. As pointed out by Kullback and Leibler \cite{Kullback1951}
the notion of sufficiency is closely related to the notion of information
divergence (or Kullback-Leibler divergence). As we shall see knowledge
of the second moment is asymptotically sufficient in the Law of Thin
Numbers. Jaynes developed the Maximum Entropy Principle where entropy
was maximized under some constraints \cite{Jaynes1957}. An obvious
problem about the Maximum Entropy Principle is how to determine which
constraints are relevant for a specific problem. In thermodynamics
experience of generations of physicists has shown that for an isolated
gas in thermodynamic equilibrium the pressure and temperature are
sufficient in the sense that knowing only these two quantities allow
you to determine any other physical property via the maximum entropy
principle. The problem of determining which statistic is relevant,
persists when the Maximum Entropy Principle is replaced by a Minimum
Information Principle relative to some prior distribution. The results
of this paper may be viewed as a systematic approach to the problem
of finding which quantities are sufficient for Poisson approximation.

The rest of the paper is organized as follows. In Section \ref{SecMoment}
Poisson-Charlier polynomials are introduced to simplify moment calculations
for thinned sums of independent random variables. Since one of our
main techniques is based on information projections under moment constraints
we have to address the problem of the existence of information projections
in Section \ref{sec:mininf}. These results may be of independent
interest. In Section \ref{SecLower} we state our main results on
sharp lower bounds on information divergence. Many of our calculations
involve Poisson-Charlier polynomials and are quite lengthy. Proofs
are postponed to the appendix. In some cases we are not able to get
sharp lower bounds but under weak conditions the lower bounds are
still asymptotically correct as stated in Section \ref{sec:lb}. Our
results are related to the Central Limit Theorem and the Gaussian
distribution in Section \ref{sec:disccont}. We end with the conclusion
followed by the appendix containing several of the proofs.

\section{Inequalities in exponential families}

Let $\beta\rightarrow Q_{\beta},\beta\in\Gamma$ denote an exponential
family such that the Radon-Nikodym derivative is 
\[
\frac{\mathrm{d}Q_{\beta}}{\mathrm{d}Q_{0}}=\frac{\exp\left(\beta\cdot x\right)}{Z\left(\beta\right)}
\]
 and where $\Gamma$ is the set of $\beta$ such that the \emph{partition
function} $Z$ is finite, i.e.
\[
Z\left(\beta\right)=\int\exp\left(\beta\cdot x\right)\,\mathrm{d}Q_{0}\left(x\right)<\infty.
\]
 The partition function is also called the \emph{moment generating
function}. The parametrization $\beta\rightarrow Q_{\beta}$ is called
the \emph{natural parametrization}. The mean value of the distribution
$Q_{\beta}$ will be denoted $\mu_{\beta}.$ The distribution with
mean value $\mu$ is denoted $Q^{\mu}$ so that $Q^{\mu_{\beta}}=Q_{\beta}.$
The inverse of the function $\beta\rightarrow\mu_{\beta}$ is denoted
$\hat{\beta}\left(\cdot\right)$ and equals the maximum likelihood
estimate of the natural parameter. The variance of $Q^{\mu}$ is denoted
$V\left(\mu\right)$ so that $\mu\rightarrow V\left(\mu\right)$ is
the \emph{variance function} of the exponential family. This variance
function uniquely characterizes the exponential family.

We note that $\beta\rightarrow\ln Z\left(\beta\right)$ is the \emph{cumulant
generating function} so that
\begin{align*}
\frac{\mathrm{d}}{d\beta}\ln Z\left(\beta\right)_{\mid\beta=0} & =\mathrm{{E}}\left[X\right],\\
\frac{\mathrm{d}^{2}}{\mathrm{d}\beta^{2}}\ln Z\left(\beta\right)_{\mid\beta=0} & =\mathrm{{Var}}\left(X\right),\\
\frac{\mathrm{d}^{3}}{\mathrm{d}\beta^{3}}\ln Z\left(\beta\right)_{\mid\beta=0} & =\mathrm{{E}}\left[\left(X-\mathrm{{E}}\left[X\right]\right)^{3}\right].
\end{align*}

\begin{lem}
\label{PeterG}Let $\beta\rightarrow Q_{\beta},\beta\in\Gamma$ denote
an exponential family with 
\[
\frac{\mathrm{d}Q_{\beta}}{\mathrm{d}Q_{0}}=\frac{\exp\left(\beta\cdot x\right)}{Z\left(\beta\right)}.
\]
 Then\end{lem}
\begin{enumerate}
\item for some $\gamma$ between $\alpha$ and $\beta,$ 
\[
D\left(Q_{\alpha}\Vert Q_{\beta}\right)=\frac{V\left(\mu_{\gamma}\right)}{2}\left(\alpha-\beta\right)^{2}
\]
 and
\item for some $\eta$ between $\mu$ and $\nu.$
\[
D\left(Q^{\mu}\Vert Q^{\nu}\right)=\frac{\left(\mu-\nu\right)^{2}}{2V\left(\eta\right)}.
\]
\end{enumerate}
\begin{proof}
The two parts of the theorem are proved separately.

1. We consider the function 
\begin{align*}
f\left(s\right) & =D\left(Q_{\alpha}\Vert Q_{s}\right)\\
 & =\int\ln\frac{\mathrm{d}Q_{\alpha}}{\mathrm{d}Q_{s}}\,\mathrm{d}Q_{\alpha}\\
 & =\int\ln\frac{\frac{\exp\left(\alpha\cdot x\right)}{Z\left(\alpha\right)}}{\frac{\exp\left(s\cdot x\right)}{Z\left(s\right)}}\,\mathrm{d}Q_{\alpha}\left(x\right)\\
 & =\int\left(\left(\alpha-s\right)\cdot x+\ln Z\left(s\right)-\ln Z\left(\alpha\right)\right)\,\mathrm{d}Q_{\alpha}\\
 & =\left(\alpha-s\right)\cdot\mu_{\alpha}+\ln Z\left(s\right)-\ln Z\left(\alpha\right)
\end{align*}
 The derivatives of this function are
\begin{align*}
f^{\prime}\left(s\right) & =-\mu_{\alpha}+\frac{Z^{\prime}\left(s\right)}{Z\left(s\right)}=\mu_{s}-\mu_{\alpha}~,\\
f^{\prime\prime}\left(s\right) & =\frac{Z\left(s\right)Z^{\prime\prime}\left(s\right)-\left(Z^{\prime}\left(s\right)\right)^{2}}{Z\left(s\right)^{2}}=V\left(\mu_{s}\right).
\end{align*}
 According to Taylor's formula there exists $\gamma$ between $\alpha$
and $\beta$ such that
\begin{align*}
D & \left(Q_{\alpha}\Vert Q_{\beta}\right)\\
 & =f\left(\alpha\right)+\left(\beta-\alpha\right)f^{\prime}\left(\alpha\right)+\frac{1}{2}\left(\beta-\alpha\right)^{2}f^{\prime\prime}\left(\gamma\right)\\
 & =\frac{V\left(\mu_{\gamma}\right)}{2}\left(\beta-\alpha\right)^{2}.
\end{align*}

2. We consider the function 
\begin{align*}
g\left(t\right) & =D\left(Q^{t}\Vert Q^{\nu}\right)\\
 & =\left(\hat{\beta}\left(t\right)-\hat{\beta}\left(\nu\right)\right)\cdot t\\
 & \qquad+\ln Z\left(\hat{\beta}\left(\nu\right)\right)-\ln Z\left(\hat{\beta}\left(t\right)\right).
\end{align*}
 The derivatives of this function are
\begin{align*}
g^{\prime}\left(t\right) & =\frac{\mathrm{d}\hat{\beta}\left(t\right)}{dt}t+\left(\hat{\beta}\left(t\right)-\hat{\beta}\left(\nu\right)\right)\\
 & \qquad-\frac{Z^{\prime}\left(\hat{\beta}\left(t\right)\right)}{Z\left(\hat{\beta}\left(t\right)\right)}\frac{\mathrm{d}\hat{\beta}\left(t\right)}{\mathrm{d}t}\\
 & =\hat{\beta}\left(t\right)-\hat{\beta}\left(\nu\right),\\
g^{\prime\prime}\left(t\right) & =\frac{1}{\mathrm{d}t/d\hat{\beta}\left(t\right)}=\frac{1}{V\left(t\right)}.
\end{align*}
 According to Taylor's formula there exists $\eta$ between $\mu$
and $\nu$ such that
\begin{align*}
D & \left(Q^{\mu}\Vert Q^{\nu}\right)\\
 & =g\left(\nu\right)+\left(\mu-\nu\right)f^{\prime}\left(\nu\right)+\frac{1}{2}\left(\mu-\nu\right)^{2}f^{\prime\prime}\left(\eta\right)\\
 & =\frac{\left(\mu-\nu\right)^{2}}{2V\left(\eta\right)}.
\end{align*}
\end{proof}
\begin{defn}
The signed log-likelihood is defined by
\[
G\left(\mu\right)=\begin{cases}
-\left(2\cdot D\left(\left.Q^{\mu}\right\Vert Q_{0}\right)\right)^{\nicefrac{1}{2}} & \mu\leq\nu,\\
\left(2\cdot D\left(\left.Q^{\mu}\right\Vert Q_{0}\right)\right)^{\nicefrac{1}{2}} & \mu>\nu.
\end{cases}
\]
\end{defn}
\begin{prop}
\label{prop:SignedLogLikelihood}Let $\mu_{0}$denote the mean value
of $Q_{0}$ and let $\mu$ denote some other possible mean value.
Then for some $\eta$ between $\mu$ and $\nu.$
\[
G\left(\mu\right)=\frac{\mu-\mu_{0}}{V\left(\eta\right)^{\nicefrac{1}{2}}}.
\]
\end{prop}
\begin{lem}
\label{LemmaNedreEpsilon}Let $\beta\rightarrow Q_{\beta},\beta\in\Gamma$
denote an exponential family with 
\[
\frac{\mathrm{d}Q_{\beta}}{\mathrm{d}Q_{0}}=\frac{\exp\left(\beta\cdot x\right)}{Z\left(\beta\right)}.
\]
 If $\mu_{0}=0$ and $V\left(0\right)=1$ and 
\[
\mathrm{{E}}_{Q_{0}}\left[X^{3}\right]>0
\]
then $G\left(\mu\right)\leq\mu$ holds for $\mu$ in a neighborhood of 0. \end{lem}
\begin{proof}
From Proposition \ref{prop:SignedLogLikelihood} we know that there
exists $\eta$ between $\mu$ and 0 such that 
\begin{align*}
G\left(\mu\right) & =\frac{\mu-0}{V\left(\eta\right)^{\nicefrac{1}{2}}}.
\end{align*}
 Therefore it is sufficient to prove that $V\left(\eta\right)$ is
increasing in a neighborhood of 0. This follows because 
\begin{align*}
\frac{\mathrm{d}V\left(\eta\right)}{\mathrm{d}\eta} & =\frac{\frac{\mathrm{d}V\left(\eta\right)}{\mathrm{d}\beta}}{\frac{\mathrm{d}\eta}{\mathrm{d}\beta}}\\
 & =\frac{\frac{\mathrm{d}^{3}}{\mathrm{d}\beta^{3}}\ln Z\left(\beta\right)}{\frac{\mathrm{d}^{2}}{\mathrm{d}\beta^{2}}\ln Z\left(\beta\right)}\\
 & =\frac{\mathrm{\mathrm{{E}}}\left[\left(X-\eta\right)^{3}\right]}{\mathrm{Var}\left(X\right)}
\end{align*}
where the mean and variance are taken with respect to the element
in the exponential family with mean $\eta.$ Since $\frac{\mathrm{E}\left[X^{3}\right]}{\mathrm{Var}\left(X\right)}>0$
for $\beta=0$ we have that $\frac{\mathrm{E}\left[\left(X-\eta\right)^{3}\right]}{\mathrm{Var}\left(X\right)}>0$
for $\beta$ in a neighborhood of 0 so $V$ is increasing.
\end{proof}

\section{Moment calculations\label{SecMoment}}

We shall use the notation $x^{\underline{k}}=x\left(x-1\right)\cdots\left(x-k+1\right)$
for the falling factorial. The factorial moments of an $\alpha$-thinning
are easy to calculate
\begin{align}
\mathrm{{E}}\left[\left(\frac{1}{n}\circ Y\right)^{\underline{k}}\right] & =\mathrm{{E}}\left[\left(\sum_{n=1}^{Y}X_{n}\right)^{\underline{k}}\right]\label{Factorial}\\
 & =\mathrm{{E}}\left[\mathrm{{E}}\left[\left.\left(\sum_{n=1}^{Y}X_{n}\right)^{\underline{k}}\right\vert Y\right]\right]\\
 & =\mathrm{{E}}\left[\alpha^{k}Y^{\underline{k}}\right]=\alpha^{k}\mathrm{{E}}\left[Y^{\underline{k}}\right].\nonumber 
\end{align}
 Thus, thinning scales the factorial moments in the same way as ordinary
multiplication scales the ordinary moments.

The binomial distributions, Poisson distributions, geometric distributions,
negative binomial distributions, inverse binomial distributions, and
generalized Poisson distributions are exponential families with at
most cubic variance functions \cite{Morris1982}\cite{Letac1990}.
The thinned family is also exponential.
\begin{thm}
Let $V$ be the variance function of an exponential family with $X\in\mathbb{N}_{0}$
as sufficient statistics and let $V_{\alpha}$ denote the variance
function of the $\alpha$-thinned family. Then the variance functions
$V$ and $V_{a}$ are related by the equation
\[
V_{\alpha}\left(x\right)=\alpha^{2}V\left(\frac{x}{\alpha}\right)-\alpha x+x.
\]
\end{thm}
\begin{proof}
The variance of a thinned variable can be calculated as 
\begin{align*}
\mathrm{Var}\left(\alpha\circ X\right) & =\mathrm{{E}}\left[\left(\alpha\circ X\right)^{\underline{2}}\right]+\mathrm{{E}}\left[\alpha\circ X\right]-\left(\mathrm{{E}}\left[\alpha\circ X\right]\right)^{2}\\
 & =\alpha^{2}\mathrm{{E}}\left[X^{\underline{2}}\right]+\mathrm{{E}}\left[\alpha\circ X\right]-\left(\mathrm{{E}}\left[\alpha\circ X\right]\right)^{2}\\
 & =\alpha^{2}\left(\mathrm{Var}\left(X\right)-\mathrm{{E}}\left[X\right]+\mathrm{{E}}\left[X\right]^{2}\right)\\
 & \qquad+\mathrm{{E}}\left[\alpha\circ X\right]-\left(\mathrm{{E}}\left[\alpha\circ X\right]\right)^{2}\\
 & =\alpha^{2}V\left(\frac{\mathrm{{E}}\left[\alpha\circ X\right]}{\alpha}\right)\\
 & \qquad-\alpha\mathrm{{E}}\left[\alpha\circ X\right]+\mathrm{{E}}\left[\alpha\circ X\right].
\end{align*}

\end{proof}
For instance the variance function of the Poisson distributions is
$V\left(x\right)=x$ and therefore $V_{\alpha}\left(x\right)=\alpha^{2}V\left(\frac{x}{\alpha}\right)-\alpha x+x=x$
so the thinned family is also Poisson. In general the thinned family
has a variance function that is a polynomial of the same order and
structure. Therefore not only the Poisson family but all the above
mentioned families are conserved under thinning. This kind of variance
function calculations can also be used to verify that if $V$ is the
variance function of the exponential family based on $P$ then the
variance function of the exponential family based on $\frac{1}{n}\circ P^{\ast n}$
is
\[
x\rightarrow\frac{V\left(x\right)-x}{n}+x.
\]
In particular the variance function corresponding to a thinned sum
converges to the variance function of the Poisson distributions. This
observation can be used to give an alternative proof of the law of
thin numbers, but we shall not develop this idea any further in the
present paper.

For moment calculations involving sums of thinned variables we shall
also use the \emph{Poisson-Charlier polynomials} \cite{Chihara1978},
which are given by
\[
C_{k}^{\lambda}\left(x\right)=\left(\lambda^{k}k!\right)^{\nicefrac{-1}{2}}\sum_{\ell=0}^{k}\binom{k}{\ell}\left(-\lambda\right)^{k-\ell}x^{\underline{\ell}}
\]
 where $k\in\mathbb{N}_{0}.$ The Poisson-Charlier polynomials are
characterized as normalized orthogonal polynomials with respect to
the Poisson distribution $\mathrm{{Po}}\left(\lambda\right).$ The
first three Poisson-Charlier polynomials are 
\begin{align*}
C_{0}^{\lambda}\left(x\right) & =1,\\
C_{1}^{\lambda}\left(x\right) & =\frac{x-\lambda}{\lambda^{\nicefrac{1}{2}}},\\
C_{2}^{\lambda}\left(x\right) & =\frac{x^{2}-\left(2\lambda+1\right)x+\lambda^{2}}{2^{\nicefrac{1}{2}}\lambda}.
\end{align*}
 A mean value of a Poisson-Charlier polynomial will be called a\emph{
Poisson-Charlier moment}. First we note that if $\mathrm{{E}}\left[X\right]=\lambda$
the second Poisson-Charlier moment is given by
\[
\mathrm{{E}}\left[C_{2}^{\lambda}\left(x\right)\right]=\frac{\mathrm{{Var}}\left(X\right)-\lambda}{2^{\nicefrac{1}{2}}\lambda}.
\]

Let $\kappa$ denote the first value of $k$ such that $\mathrm{{E}}\left[C_{k}^{\lambda}(X)\right]\neq0\ $
or, equivalently, $\mathrm{{E}}\left[X^{\underline{k}}\right]\neq\lambda^{k}.$
Lower bounds on the rate of convergence in the thin law of large numbers
are essentially given in terms of $\kappa$ and $c=\mathrm{{E}}\left[C_{\kappa}^{\lambda}(X)\right].$
\begin{prop}
\label{PropositionPCDecay}The Poisson-Charlier moments satisfy 
\[
\mathrm{{E}}\left[C_{k}^{\lambda}\left(\frac{1}{n}\circ\sum_{j=1}^{n}X_{j}\right)\right]=\frac{\mathrm{{E}}[C_{k}^{\lambda}(X)]}{n^{k-1}}
\]
 for $k=0,1,\dots,\kappa+1.$ \end{prop}
\begin{proof}
This follows by straightforward calculations based on Lemma \ref{thm:mombeh}
that can be found in the appendix. 
\end{proof}
For some moment calculations the following result of Khoklov is useful.
\begin{lem}
{[}Khoklov \cite{Khokhlov2002}{]}%
\footnote{The original formula in \cite{Khokhlov2002} contains a typo in that
the factor $\left(-1\right)^{k+l}$ is missing but the proof is correct.%
}\label{Khokhlovform}
\[
C_{k}^{\lambda}\left(x\right)C_{\ell}^{\lambda}\left(x\right)=\left(-1\right)^{k+\ell}\left(\frac{\lambda^{k+\ell}}{k!\ell!}\right)^{\nicefrac{1}{2}}\sum_{m=0}^{k+\ell}c_{m}C_{m}^{\lambda}\left(x\right)
\]
 where $c_{m}$ as a function of $k,\ell$ and $\lambda$ is given
by
\[
\sum_{n=0}^{m}\frac{\sum_{\mu=0}^{k}\sum_{\nu=0}^{\ell}\binom{k}{\mu}\binom{\ell}{\nu}\mu^{\underline{n}}\nu^{\underline{n}}\left(\mu+\nu-n\right)^{\underline{m}}\left(-1\right)^{\mu+\nu}}{\left(m!\lambda^{m}\right)^{\nicefrac{1}{2}}n!\lambda^{n}}.
\]

\end{lem}
In the appendix we use Lemma \ref{Khokhlovform} to prove the following
result.
\begin{lem}
\label{Lemma:kvad}\label{LemmaEfterKhoklov}
\[
\left(C_{k}^{\lambda}\left(x\right)\right)^{2}=\frac{\lambda^{k}}{k!}\sum_{m=0}^{2k}c_{m}C_{m}^{\lambda}\left(x\right)
\]
 where $c_{m}$ as a function of $k$ and $\lambda$ is given by
\[
\left(m!\lambda^{m}\right)^{-\nicefrac{1}{2}}\sum_{n=0}^{m}\frac{\left(k^{\underline{n}}\right)^{3}n^{\underline{k-n}}}{n!\lambda^{n}}.
\]

\end{lem}

\begin{lem}
\label{LemmaKoklpos} For a Poisson random variable $X$ with mean
value $\lambda$ we have 
\[
\mathrm{{E}}\left[C_{k}^{\lambda}\left(X\right)^{3}\right]>0
\]
 for any $k\in\mathbb{N}_{0}.$ \end{lem}
\begin{proof}
According to Lemma \ref{LemmaEfterKhoklov} we have 
\begin{align*}
\mathrm{{E}}\left[\left(C_{k}^{\lambda}\left(X\right)\right)^{3}\right] & =\mathrm{{E}}\left[\left(\frac{\lambda^{k}}{k!}\sum_{m=0}^{2k}c_{m}C_{m}^{\lambda}\left(X\right)\right)C_{k}^{\lambda}\left(X\right)\right]\\
 & =\frac{\lambda^{k}}{k!}c_{k}.
\end{align*}
 where $c_{k}$ is defined in Lemma \ref{LemmaEfterKhoklov}. Therefore
it is sufficient to prove that
\[
\sum_{n=0}^{k}\frac{\left(k^{\underline{n}}\right)^{3}n^{\underline{k-n}}}{n!\lambda^{n}}>0.
\]
 This follows because $\left(k^{\underline{n}}\right)^{3}n^{\underline{k-n}}$
is always non-negative and it is positive when $\nicefrac{k}{2}\leq n\leq k.$ 
\end{proof}

\section{Existence of minimum information distributions\label{sec:mininf}}

Let $X$ be a random variable for which the moments of order $1,2,...\ell$
exist. We shall assume that $\mathrm{{E}}\left[X\right]=\lambda$.
We are interested in minimizing information divergence $D\left(X\Vert\mathrm{{Po}}\left(\lambda\right)\right)$
under linear conditions on the moments of $X$ and derive conditions
for a minimum information distribution to exist. We shall use $D\left(C\Vert\mathrm{{Po}}\left(\lambda\right)\right)$
as notation for $\inf_{P\in C}D\left(P\Vert\mathrm{{Po}}\left(\lambda\right)\right)$
\begin{lem}
\label{LemmaMinLig}For some fixed set $\left(h_{1},\cdots,h_{\ell}\right)\in\mathbb{R}^{\ell}$,
let $K$ be the convex set of distributions on $\mathbb{N}_{0}$ for
which the first $\ell$ moments are defined and which satisfies the
following conditions 
\begin{align}
\mathrm{{E}}_{P}\left[C_{k}^{\lambda}\left(X\right)\right] & =h_{k}~,\text{ for }k=1,2,\cdots,\ell-1;\label{betingelse1}\\
\mathrm{{E}}_{P}\left[C_{\ell}^{\lambda}\left(X\right)\right] & \leq h_{\ell}\;.\label{betingelse-1}
\end{align}
 If $D\left(K\Vert\mathrm{{Po}}\left(\lambda\right)\right)<\infty$
then the minimum information projection of $\mathrm{{Po}}\left(\lambda\right)$
exists. \end{lem}
\begin{proof}
Let $\vec{G}\in\mathbb{R}^{\ell-1}$ be a vector and let $C_{\vec{G}}$
be the set of distributions satisfying the following inequalities
\[
\mathrm{{E}}\left[C_{\ell}^{\lambda}\left(X\right)-h_{\ell}-\sum_{k<\ell}G_{k}\cdot\left(C_{k}^{\lambda}\left(X\right)-h_{k}\right)\right]\leq0.
\]
 We see that the set $C_{\vec{G}}$ is tight because $C_{\ell}^{\lambda}\left(x\right)-h_{\ell}-\sum_{k<\ell}G_{k}\left(C_{k}^{\lambda}\left(x\right)-h_{k}\right)\rightarrow\infty$
for $x\rightarrow\infty$. Therefore the intersection $K=\bigcap_{\vec{G}\in\mathbf{R}^{\ell-1}}C_{\vec{G}}$
is compact. There exists a distribution $P^{\ast}\in K$ such that
the information divergence $D\left(P\Vert\mathrm{{Po}}\left(\lambda\right)\right)$
is minimal because $K$ is compact. \end{proof}
\begin{thm}
\label{thm:mininf copy(1)} Let $C$ be the set of distributions on
$\mathbb{N}_{0}$ for which the first $\ell$ moments are defined
and satisfy the following equations 
\begin{equation}
\mathrm{{E}}\left[C_{k}^{\lambda}\left(X\right)\right]=h_{k}\text{ for }k=1,2,\cdots,\ell\,.\label{betingelser}
\end{equation}
 Assume that $D\left(C\Vert\mathrm{{Po}}\left(\lambda\right)\right)<\infty$
and $\ell\geq2.$ Consider the following three cases:\end{thm}
\begin{enumerate}
\item $h_{k}=0$ for $k<\ell$ and $h_{\ell}>0.$
\item $h_{k}=0$ for $k<\ell$ and $h_{\ell}<0.$
\item $h_{k}=0$ for $k<\ell-1$ and $h_{\ell-1}>0.$ 
\end{enumerate}
In case 1 no minimizer exists and $D\left(C\Vert\mathrm{Po}\left(\lambda\right)\right)=0$.
In case 2 and 3 there exists a distribution $P^{\ast}\in C$ that
minimizes $D\left(P\Vert\mathrm{{Po}}\left(\lambda\right)\right).$
\begin{proof}
\textbf{Case 1.} If a minimizer existed it would be an element of
the corresponding exponential family, but the partition function cannot
be finite because $h_{\ell}>0$ and $\ell\geq2.$

For cases 2 and 3 let $P=P^{\ast}$ be the minimum information distribution
satisfying the conditions.

\textbf{Case 2}. Assume that $h_{k}=0$ for $k<\ell$ and $h_{\ell}<0.$
Assume also that $\mathrm{{E}}_{P^{\ast}}\left[C_{\ell}^{\lambda}\left(X\right)\right]<h_{\ell}~.$
Define $P^{\theta}=\theta P^{\ast}+\left(1-\theta\right)\mathrm{{Po}}\left(\lambda\right).$
Then the conditions (\ref{betingelse1}) holds for $P=P^{\theta}$
and
\begin{align*}
\mathrm{{E}}_{P^{\theta}}\left[C_{\ell}^{\lambda}\left(X\right)\right] & =\theta\mathrm{{E}}_{P^{\ast}}\left[C_{\ell}^{\lambda}\left(X\right)\right]+\left(1-\theta\right)\mathrm{{E}}_{\mathrm{{Po}}\left(\lambda\right)}\left[C_{\ell}^{\lambda}\left(X\right)\right]\\
 & =\theta\mathrm{{E}}_{P^{\ast}}\left[C_{\ell}^{\lambda}\left(X\right)\right].
\end{align*}
 Thus $\mathrm{{E}}_{P^{\theta}}\left[C_{\ell}\left(\lambda,X\right)\right]=h_{\ell}$
if 
\[
\theta=\frac{h_{\ell}}{\mathrm{{E}}_{P^{\ast}}\left[C_{\ell}^{\lambda}\left(X\right)\right]}\in\left]0,1\right[.
\]
 Therefore $P^{\theta}$ satisfies (\ref{betingelser}) but 
\[
D\left(P^{\theta}\Vert\mathrm{{Po}}\left(\lambda\right)\right)\leq\theta D\left(P^{\ast}\Vert\mathrm{{Po}}\left(\lambda\right)\right)<D\left(P^{\ast}\Vert\mathrm{{Po}}\left(\lambda\right)\right)
\]
 and we have a contradiction.

\textbf{Case 3}. Now, assume that $h_{k}=0$ for $k<\ell-1$ and $h_{\ell-1}>0.$
Moreover, assume that $\mathrm{{E}}_{P^{\ast}}\left[C_{\ell}^{\lambda}\left(X\right)\right]<h_{\ell}$.
Using the result of case 1 we see that there exists a distribution
$\tilde{P}$ for which the $\ell$ first moments exist and that the
first $\ell-1$ moments satisfy (\ref{betingelse1}) but with $D\left(\tilde{P}\Vert\mathrm{{Po}}\left(\lambda\right)\right)<D\left(P^{\ast}\Vert\mathrm{{Po}}\left(\lambda\right)\right).$
Define $P^{\theta}=\theta P^{\ast}+\left(1-\theta\right)\tilde{P}\left(\lambda\right).$
Then the conditions (\ref{betingelse1}) holds for $P=P^{\theta}$
and
\begin{multline}
\mathrm{{E}}_{P^{\theta}}\left[C_{\ell}^{\lambda}\left(X\right)\right]=\\
\theta\mathrm{{E}}_{P^{\ast}}\left[C_{\ell}^{\lambda}\left(X\right)\right]+\left(1-\theta\right)\theta\mathrm{{E}}_{\tilde{P}}\left[C_{\ell}^{\lambda}\left(X\right)\right].
\end{multline}
 Therefore $\mathrm{{E}}_{P^{\theta}}\left[C_{\ell}^{\lambda}\left(X\right)\right]\leq h_{\ell}$
for $\theta$ sufficiently close to $1$ but $D\left(P^{\theta}\Vert\mathrm{{Po}}\left(\lambda\right)\right)\leq\theta D\left(P^{\ast}\Vert\mathrm{{Po}}\left(\lambda\right)\right)<D\left(P^{\ast}\Vert\mathrm{{Po}}\left(\lambda\right)\right)$
and we have a contradiction. Therefore $P^{\ast}$ satisfies the equation
$\mathrm{{E}}\left[C_{\ell}^{\lambda}\left(X\right)\right]=h_{\ell}.$ 
\end{proof}
For the applications we have in mind it will be easy to check the
condition 
\[
D\left(C\Vert\mathrm{{Po}}\left(\lambda\right)\right)<\infty,
\]
 but in general it may be difficult even to determine simple necessary
and sufficient conditions for $C\neq\varnothing$ in terms of a set
of specified moments.

\section{Lower bounds\label{SecLower}}

First we consider the exponential family based on the distribution
$Po\left(\lambda\right)$. The variance function of this exponential
family is $V\left(\mu\right)=\mu$ which is an increasing function.
Hence 
\[
G\left(\mbox{\ensuremath{\mu}}\right)\leq\frac{\mu-\lambda}{\lambda^{\nicefrac{1}{2}}}.
\]
 Squaring this inequality for $\mu\leq\lambda$ gives
\begin{equation}
D\left(\mathrm{{Po}}\left(\mu\right)\Vert\mathrm{{Po}}\left(\lambda\right)\right)\geq\frac{\left(\mu-\lambda\right)^{2}}{2\lambda}.\label{foerste}
\end{equation}

Let $X$ be a random variable with values in $\mathbb{N}_{0}$ and
with mean $\mu.$ Then the divergence $D\left(X\Vert\mathrm{{Po}}\left(\lambda\right)\right)$
is minimal if the distribution of $X$ is an element of the associated
exponential family, i.e.
\begin{align*}
D\left(X\Vert\mathrm{{Po}}\left(\lambda\right)\right) & \geq D\left(\mathrm{{Po}}\left(\mu\right)\Vert\mathrm{{Po}}\left(\lambda\right)\right)\\
 & \geq\frac{\mathrm{{E}}\left[C_{1}^{\lambda}\left(X\right)\right]^{2}}{2}.
\end{align*}
 We conjecture that a result similar to (\ref{foerste}) holds for
any order of the Poisson-Charlier polynomial.
\begin{conjecture}
\label{KonjMindre}For any random variable $X$ with values in $\mathbb{N}_{0}$
and for any $k\in\mathbb{N}$ we have 
\begin{equation}
D\left(X\Vert\mathrm{{Po}}\left(\lambda\right)\right)\geq\frac{\mathrm{{E}}\left[C_{k}^{\lambda}\left(X\right)\right]^{2}}{2}\ \label{conjecture}
\end{equation}
 if $\mathrm{{E}}\left[C_{k}^{\lambda}\left(X\right)\right]\leq0.$ 
\end{conjecture}
We have not been able to prove this conjecture but we can prove some
partial results.
\begin{thm}
\label{lokal}For any random variable $X$ with values in $\mathbb{N}_{0}$
and any $k\in\mathbb{N}$ there exists $\varepsilon>0$ such that
for $\mathrm{{E}}\left[C_{k}^{\lambda}\left(X\right)\right]\in\left[-\varepsilon,0\right]$
inequality (\ref{conjecture}) holds. \end{thm}
\begin{proof}
Let $Z\left(\beta\right)$ denote the partition function of the exponential
family based on $\mathrm{{Po}}\left(\lambda\right)$ with $C_{k}^{\lambda}\left(X\right)$
as sufficient statistics. The function $x\rightarrow C_{k}^{\lambda}\left(x\right)$
is lower bounded so for any $\beta\leq0$ the partition function $Z\left(b\right)$
is finite. Therefore for any $c\in\left]\min C_{k}^{\lambda}\left(x\right),0\right]$
there exists an element in the exponential family with $c$ as mean
value. This distribution minimize information divergence under the
constraint $\mathrm{{E}}\left[C_{k}^{\lambda}\left(X\right)\right]=c.$
Therefore we can use \ref{LemmaNedreEpsilon} with $X$ replaced by
$C_{k}^{\lambda}\left(X\right)$ and we just need to prove that $\mathrm{{E}}\left[C_{k}^{\lambda}\left(X\right)^{3}\right]>0,$
which is done in Lemma \ref{LemmaKoklpos} below. 
\end{proof}
Conjecture \ref{KonjMindre} can be proved for $k=2.$ The proof is
quite long in order to cover all cases so the main part of the proof
is postponed to the appendix.
\begin{thm}
\label{nedregraense}For any random variable $X$ with values in $\mathbb{N}_{0}$
the inequality (\ref{conjecture}) holds for $k=2$ if $\mathrm{{E}}\left[C_{2}^{\lambda}\left(X\right)\right]\leq0.$ \end{thm}
\begin{rem}
Note that there is no assumption on the mean of $X$ in this theorem.
Theorem \ref{ThmNedreVar} is a reformulation of Theorem \ref{nedregraense}
under the extra assumption that $\mathrm{{E}}\left[X\right]=\lambda.$ 

As in Lemma \ref{LemmaNedreEpsilon}, we want to demonstrate that
$\mathrm{Var}\left(X_{\beta}\right)\leq1.$ Since the variance is
\[
\frac{\mathrm{d}^{2}}{\mathrm{d}\beta^{2}}\ln\left(Z\left(\beta\right)\right)=\frac{Z^{\prime\prime}\left(\beta\right)Z\left(\beta\right)-\left(Z^{\prime}\left(\beta\right)\right)^{2}}{\left(Z\left(\beta\right)\right)^{2}}<\frac{Z^{\prime\prime}\left(\beta\right)}{Z\left(\beta\right)},
\]
we will prove the theorem by giving bounds on $Z^{\prime\prime}\left(\beta\right).$ \end{rem}
\begin{proof}
Define $\beta_{0}$ as the negative solution to 
\[
\beta^{2}\exp\left(\beta^{2}\right)=1.
\]
 Numerical calculations give $\beta_{0}=\textrm{-}0.753\,09\ .$ We
observe that $\beta_{0}$ is slightly less that $\textrm{-}2^{\nicefrac{\textrm{-}1}{2}}$
and divide the proof in two cases depending on the value of $\mathrm{{E}}\left[C_{2}^{\lambda}\left(X\right)\right]$
compared with $\beta_{0}.$ The minimum of the function $C_{2}^{\lambda}\left(x\right)$
with $\mathbb{R}$ as domain is $\textrm{-}2^{\nicefrac{\textrm{-}1}{2}}-1/(4\cdot2^{\nicefrac{1}{2}}\cdot\lambda).$
Hence $C_{2}^{\lambda}\left(x\right)\geq\beta_{0}$ if $\lambda\geq3.\,844\,4$
and $\mathrm{{E}}\left[C_{2}^{\lambda}\left(X\right)\right]\geq\beta_{0}.$
Note also that $\mathrm{{E}}\left[C_{2}^{\lambda}\left(X\right)\right]\geq C_{2}^{\lambda}\left(\mathrm{{E}}\left[X\right]\right)$
according to Jensen's Inequality. The Poisson-Charlier polynomial
of order 2 satisfies 
\[
C_{2}^{\lambda}\left(\lambda\right)=C_{2}^{\lambda}\left(\lambda+1\right)=\textrm{-}2^{\nicefrac{\textrm{-}1}{2}}
\]
so $C_{2}^{\lambda}\left(x\right)\geq\textrm{-}2^{\nicefrac{\textrm{-}1}{2}}$
for $x\leq\lambda$ and for $x\geq\lambda+1.$ Hence $\mathrm{{E}}\left[C_{2}^{\lambda}\left(X\right)\right]\geq\textrm{-}2^{\nicefrac{\textrm{-}1}{2}}$
if $\mathrm{{E}}\left[X\right]\leq\lambda$ or if $\mathrm{{E}}\left[X\right]\geq\lambda+1.$
We see $\mathrm{{E}}\left[C_{2}^{\lambda}\left(X\right)\right]\geq\beta_{0}$
except if $X$ is very concentrated near $\lambda.$

\textbf{Case 1 where} $\mathrm{{E}}\left[C_{2}^{\lambda}\left(X\right)\right]\geq\beta_{0}.$
Let ${Po}\left(\lambda\right)_{\beta}$ be the exponential family
based on ${Po}\left(\lambda\right)$ with $C_{2}^{\lambda}\left(x\right)$
as sufficient statistic and partition functions
\[
Z\left(\beta\right)=\sum_{x=0}^{\infty}\exp\left(\beta C_{2}^{\lambda}\left(x\right)\right)\mathrm{{Po}}\left(\lambda,x\right).
\]
Let ${Po}\left(\lambda\right)^{t}$ denote the element in this exponential
family with mean $t.$ We want to prove that 
\[
D\left(\left.{Po}\left(\lambda\right)^{t}\right\Vert {Po}\left(\lambda\right)\right)\geq\frac{1}{2}t^{2}
\]
 for $t\in\left[\beta_{0},0\right].$ The inequality is obvious for
$t=0$ is is sufficient to prove that
\[
\frac{\mathrm{d}D}{\mathrm{d}t}\leq t
\]
 for $t\in\left[\beta_{0},0\right].$ As in the proof of Lemma \ref{PeterG}
we have $\frac{\mathrm{d}D}{\mathrm{d}t}=\hat{\beta}\left(t\right)$
so is is sufficient to prove that $\hat{\beta}\left(t\right)\leq t$
for $t\in\left[\beta_{0},0\right].$ If $\hat{\beta}\left(t\right)\leq\beta_{0}$
this is obvious so we just have to prove that 
\[
\beta\leq t_{\beta}=\frac{Z'\left(\beta\right)}{Z\left(\beta\right)}
\]
 for $\beta\in\left[\beta_{0},0\right].$ This inequality is fulfilled
for $\beta=0$ so we differentiate once more and we just have to prove
that 
\[
1\geq\frac{Z\left(\beta\right)Z''\left(\beta\right)-\left(Z'\left(\beta\right)\right)^{2}}{\left(Z\left(\beta\right)\right)^{2}}
\]

We observe that $Z\left(0\right)=1$ and that $Z\left(\beta\right)>0.$
Since $Z^{\prime}\left(0\right)=0$ and $Z^{\prime\prime}\left(\beta\right)>0$
we have $Z\left(\beta\right)\geq1$ for all values of $\beta.$ Therefore
is sufficient to prove that $1\geq Z''\left(\beta\right)$ for $\beta\in\left[\beta_{0},0\right].$
The function 
\[
Z''\left(\beta\right)=\sum_{x=0}^{\infty}C_{2}^{\lambda}\left(x\right)^{2}\exp\left(\beta C_{2}^{\lambda}\left(x\right)\right)\mathrm{{Po}}\left(\lambda,x\right)
\]
 is convex, so it is sufficient to prove the inequality $\frac{\mathrm{d}^{2}Z}{\mathrm{d}\beta^{2}}\leq1$
for $\beta=0$ and $\beta=\beta_{0}.$ Consider the function $f\left(x\right)=x^{2}\exp\left(\beta_{0}x\right)$
with $f^{\prime}\left(x\right)=\left(2+\beta_{0}x\right)x\exp\left(\beta_{0}x\right).$
The function $f$ is decreasing for $x\leq0,$ has minimum for $x=0,$
is increasing for $0\leq x\leq\nicefrac{\textrm{-}2}{\beta_{0}},$
has a local maximum $4\exp\left(\textrm{-}2\right)/\beta_{0}^{2}$
in $x=\nicefrac{\textrm{-}2}{\beta_{0}},$ and decreasing for $x\geq\nicefrac{\textrm{-}2}{\beta_{0}}.$
Then the local maximum at $x=\nicefrac{\textrm{-}2}{\beta_{0}}$ has
value $0.954\,51<1.$ Hence $f\left(x\right)\leq1$ for $x\geq\beta_{0}.$
The minimum of the function $C_{2}^{\lambda}\left(x\right)$ with
$\mathbb{R}$ as domain is $\textrm{-}2^{\nicefrac{-1}{2}}-1/(4\cdot2^{\nicefrac{1}{2}}\cdot\lambda).$
Hence $C_{2}^{\lambda}\left(x\right)\geq\beta_{0}$ if $\lambda\geq3.\,844\,4$
so for such $\lambda,$ we have $f\left(C_{2}^{\lambda}\left(x\right)\right)\leq1$
for all $x.$ 
\begin{figure}[h]
\begin{centering}
\begin{pgfpicture}{19.23mm}{4.35mm}{104.90mm}{66.26mm} 
\pgfsetxvec{\pgfpoint{1.00mm}{0mm}} 
\pgfsetyvec{\pgfpoint{0mm}{1.00mm}} 
\color[rgb]{0,0,0}\pgfsetlinewidth{0.30mm}\pgfsetdash{}{0mm} \pgfmoveto{\pgfxy(34.14,14.67)}\pgflineto{\pgfxy(34.14,64.26)}\pgfstroke \pgfmoveto{\pgfxy(34.14,64.26)}\pgflineto{\pgfxy(33.44,61.46)}\pgflineto{\pgfxy(34.14,62.86)}\pgflineto{\pgfxy(34.84,61.46)}\pgflineto{\pgfxy(34.14,64.26)}\pgfclosepath\pgffill \pgfmoveto{\pgfxy(34.14,64.26)}\pgflineto{\pgfxy(33.44,61.46)}\pgflineto{\pgfxy(34.14,62.86)}\pgflineto{\pgfxy(34.84,61.46)}\pgflineto{\pgfxy(34.14,64.26)}\pgfclosepath\pgfstroke 
\pgfmoveto{\pgfxy(21.23,15.28)}\pgflineto{\pgfxy(102.90,15.28)}\pgfstroke \pgfmoveto{\pgfxy(102.90,15.28)}\pgflineto{\pgfxy(100.10,15.98)}\pgflineto{\pgfxy(101.50,15.28)}\pgflineto{\pgfxy(100.10,14.58)}\pgflineto{\pgfxy(102.90,15.28)}\pgfclosepath\pgffill \pgfmoveto{\pgfxy(102.90,15.28)}\pgflineto{\pgfxy(100.10,15.98)}\pgflineto{\pgfxy(101.50,15.28)}\pgflineto{\pgfxy(100.10,14.58)}\pgflineto{\pgfxy(102.90,15.28)}\pgfclosepath\pgfstroke 
\pgfputat{\pgfxy(94.09,11.32)}{\pgfbox[bottom,left]{5}} \pgfsetlinewidth{0.12mm}\pgfmoveto{\pgfxy(94.76,15.87)}\pgflineto{\pgfxy(94.76,14.66)}\pgfstroke \pgfmoveto{\pgfxy(82.67,15.87)}\pgflineto{\pgfxy(82.67,14.66)}\pgfstroke \pgfmoveto{\pgfxy(58.44,15.87)}\pgflineto{\pgfxy(58.44,14.66)}\pgfstroke 
\pgfputat{\pgfxy(45.14,11.23)}{\pgfbox[bottom,left]{1}} 
\pgfmoveto{\pgfxy(22.13,16.03)}\pgflineto{\pgfxy(22.13,14.73)}\pgfstroke \pgfmoveto{\pgfxy(25.65,15.87)}\pgflineto{\pgfxy(25.65,14.66)}\pgfstroke 
\pgfputat{\pgfxy(30.70,57.97)}{\pgfbox[bottom,left]{2}} 
\pgfmoveto{\pgfxy(34.71,59.04)}\pgflineto{\pgfxy(33.50,59.04)}\pgfstroke 
\pgfputat{\pgfxy(28.77,47.03)}{\pgfbox[bottom,left]{1.5}} \pgfmoveto{\pgfxy(34.71,48.10)}\pgflineto{\pgfxy(33.50,48.10)}\pgfstroke 
\pgfputat{\pgfxy(32.02,33.97)}{\pgfbox[bottom,left]{1}} 
\pgfputat{\pgfxy(28.77,25.15)}{\pgfbox[bottom,left]{0.5}} \pgfmoveto{\pgfxy(34.71,26.22)}\pgflineto{\pgfxy(33.50,26.22)}\pgfstroke \pgfmoveto{\pgfxy(34.71,15.27)}\pgflineto{\pgfxy(33.50,15.27)}\pgfstroke \pgfsetlinewidth{0.18mm}\pgfmoveto{\pgfxy(33.74,15.29)}\pgflineto{\pgfxy(34.48,15.29)}\pgfstroke \pgfsetlinewidth{0.30mm}\pgfmoveto{\pgfxy(21.98,61.75)}\pgfcurveto{\pgfxy(22.16,59.93)}{\pgfxy(22.53,56.28)}{\pgfxy(22.72,54.46)}\pgfcurveto{\pgfxy(22.90,52.85)}{\pgfxy(23.27,49.64)}{\pgfxy(23.45,48.03)}\pgfcurveto{\pgfxy(23.88,44.56)}{\pgfxy(24.38,40.94)}{\pgfxy(24.92,37.50)}\pgfcurveto{\pgfxy(25.97,31.02)}{\pgfxy(27.05,24.27)}{\pgfxy(30.06,18.40)}\pgfcurveto{\pgfxy(30.74,17.23)}{\pgfxy(31.69,15.94)}{\pgfxy(33.00,15.46)}\pgfcurveto{\pgfxy(34.53,14.91)}{\pgfxy(36.20,15.70)}{\pgfxy(37.42,16.60)}\pgfcurveto{\pgfxy(42.51,20.77)}{\pgfxy(46.26,26.31)}{\pgfxy(51.39,30.46)}\pgfcurveto{\pgfxy(56.93,34.96)}{\pgfxy(63.41,36.90)}{\pgfxy(70.50,35.84)}\pgfcurveto{\pgfxy(78.88,34.60)}{\pgfxy(86.88,30.91)}{\pgfxy(94.76,27.94)}\pgfstroke \color[rgb]{1,0,0}\pgfsetdash{{2.00mm}{1.00mm}}{0mm}\pgfmoveto{\pgfxy(21.98,37.16)}\pgflineto{\pgfxy(99.88,37.16)}\pgfstroke \color[rgb]{0,0,0}
\pgfputat{\pgfxy(35.88,61.05)}{\pgfbox[bottom,left]{$f(x)$}} 
\pgfputat{\pgfxy(100.81,18.10)}{\pgfbox[bottom,left]{$x$}} \pgfmoveto{\pgfxy(24.91,37.16)}\pgflineto{\pgfxy(24.91,14.74)}\pgfstroke 
\pgfputat{\pgfxy(22.79,11.70)}{\pgfbox[bottom,left]{$\beta_{0}$}} \pgfmoveto{\pgfxy(66.78,36.07)}\pgflineto{\pgfxy(66.61,15.30)}\pgfstroke 
\pgfputat{\pgfxy(62.86,11.67)}{\pgfbox[bottom,left]{$\textrm{-}\frac{2}{\beta_{0}}$}} \pgfsetdash{}{0mm}\pgfsetlinewidth{0.12mm}\pgfmoveto{\pgfxy(70.53,15.87)}\pgflineto{\pgfxy(70.53,14.66)}\pgfstroke \pgfmoveto{\pgfxy(46.26,16.03)}\pgflineto{\pgfxy(46.26,14.73)}\pgfstroke 
\pgfputat{\pgfxy(29.22,6.94)}{\pgfbox[bottom,left]{ $\textrm{-}2^{\textrm{-}\nicefrac{1}{2}}$}} 
\pgfsetlinewidth{0.15mm}\pgfmoveto{\pgfxy(30.47,9.40)}\pgflineto{\pgfxy(25.72,14.62)}\pgfstroke \pgfmoveto{\pgfxy(25.72,14.62)}\pgflineto{\pgfxy(27.08,12.08)}\pgflineto{\pgfxy(25.72,14.62)}\pgflineto{\pgfxy(28.12,13.02)}\pgflineto{\pgfxy(25.72,14.62)}\pgfclosepath\pgffill \pgfmoveto{\pgfxy(25.72,14.62)}\pgflineto{\pgfxy(27.08,12.08)}\pgflineto{\pgfxy(25.72,14.62)}\pgflineto{\pgfxy(28.12,13.02)}\pgflineto{\pgfxy(25.72,14.62)}\pgfclosepath\pgfstroke 
\end{pgfpicture}%
\par\end{centering}

\protect\caption{Plot of the function $f\left(x\right)=x^{2}\exp\left(\beta_{0}x\right).$}
\end{figure}

The graph of $x\rightarrow C_{2}^{\lambda}\left(x\right)$ is a parabola
and we have $C_{2}^{\lambda}\left(\lambda\right)=C_{2}^{\lambda}\left(\lambda+1\right)=\textrm{-}2^{\nicefrac{\textrm{-}1}{2}}.$
For all values $x\not\in\left]\lambda,\lambda+1\right[$ we have $f\left(C_{2}^{\lambda}\left(x\right)\right)\leq1.$
Since $x$ can only assume integer values only $x=\left\lceil \lambda\right\rceil $
will contribute to the mean value with a value greater than 1. A careful
inspection of different cases will show that although $x=\left\lceil \lambda\right\rceil $
will contribute to the mean with a value $f\left(x\right)>1,$ this
contribution will be averaged out with some other value of $x.$ The
details can be found in the appendix.

\textbf{Case 2 where} $\mathrm{{E}}\left[C_{2}^{\lambda}\left(X\right)\right]<\beta_{0}.$
In this case the distribution $P$ of $X$ is strongly concentrated
near $\lambda+\nicefrac{1}{2}$ where the minimum of $C_{2}^{\lambda}\left(x\right)$
is attained. According to Pinsker's inequality
\[
D\left(P\Vert\mathrm{Po}\left(\lambda\right)\right)\geq\frac{1}{2}\left\Vert P-\mathrm{Po}\left(\lambda\right)\right\Vert ^{2}
\]
 so it is sufficient to prove that 
\begin{equation}
\left\Vert P-\mathrm{Po}\left(\lambda\right)\right\Vert \geq\left\vert \mathrm{E}\left[C_{2}^{\lambda}\left(X\right)\right]\right\vert .\label{eqn:total}
\end{equation}
Again, we have to divide in a number of cases. This is done in the
appendix. 
\end{proof}

\section{Asymptotic lower bounds\label{sec:lb}}

This section combines results from Section \ref{SecMoment}, \ref{sec:mininf},
and \ref{SecLower}.

We now present lower bounds on the rate of convergence in the Law
of Thin Numbers in the sense of information divergence. The key idea
is that we bound $D(P\Vert\mathrm{{Po}}\left(\lambda\right))\geq D(P^{\ast}\Vert\mathrm{{Po}}\left(\lambda\right))$,
where $P^{\ast}$ is the minimum information distribution in a class
containing $P$. Using the construction for $P^{\ast}$ found in Section
\ref{sec:mininf}, we can find an explicit expression for the right
hand side
\begin{thm}
\label{thm:dlb} Let $X$ be a random variable with values in $\mathbb{N}_{0}.$
If $\mathrm{{E}}\left[C_{\kappa}^{\lambda}\left(X\right)\right]\leq0$
then there exists $n_{0}$ such that 
\begin{equation}
n^{2\kappa-2}D\left(\left.\frac{1}{n}\circ\sum_{j=1}^{n}X_{j}\right\Vert \mathrm{{Po}}\left(\lambda\right)\right)\geq\frac{\mathrm{{E}}\left[C_{\kappa}^{\lambda}\left(X\right)\right]^{2}}{2}.\label{lower}
\end{equation}
 for $n\geq n_{0}.$ \end{thm}
\begin{proof}
For $k=\kappa$ there exists $\varepsilon>0$ such that inequality
(\ref{conjecture}) holds when the condition in Theorem \ref{lokal}
is fulfilled. Now, by Proposition \ref{PropositionPCDecay}
\[
\mathrm{{E}}\left[C_{\kappa}^{\lambda}\left(\frac{1}{n}\circ\sum_{j=1}^{n}X_{j}\right)\right]=\frac{\mathrm{{E}}[C_{\kappa}^{\lambda}(X)]}{n^{\kappa-1}}\in\left[-\varepsilon,0\right]
\]
 for sufficiently large values of $n$ implying that 
\[
D\left(\left.\frac{1}{n}\circ\sum_{j=1}^{n}X_{j}\right\Vert \mathrm{{Po}}\left(\lambda\right)\right)\geq\frac{1}{2}\left(\frac{\mathrm{{E}}[C_{\kappa}^{\lambda}(X)]}{n^{\kappa-1}}\right)^{2},
\]
 and the lower bound (\ref{lower}) follows. 
\end{proof}
If the distribution of $X$ is ultra log-concave we automatically
have $\mathrm{{E}}\left[C_{\kappa}^{\lambda}\left(X\right)\right]\leq0,$
and we conjecture that the asymptotic lower bound is tight for ultra
log-concave distributions.

A similar lower bound on rate of convergence can be achieved even
if $\mathrm{{E}}\left[C_{\kappa}^{\lambda}\left(X\right)\right]>0,$
but then it requires the existence of a moment of higher order to
``stabilize\textquotedblright{} the moment of order $\kappa.$ Thus
we shall assume the existence of moments of all orders less than or
equal to $\kappa+1.$
\begin{thm}
If $X_{1},X_{2},\ldots$ is a sequence of independent identically
distributed discrete random variables for which all Poisson-Charlier
moments of order less than $\kappa$ are zero and for which moments
of order up to $\kappa+1$ exists then
\[
\liminf_{n\rightarrow\infty}n^{2\kappa-2}D\left(\left.\frac{1}{n}\circ\sum_{j=1}^{n}X_{j}\right\Vert \mathrm{{Po}}\left(\lambda\right)\right)\geq\frac{\mathrm{{E}}\left[C_{\kappa}^{\lambda}\left(X\right)\right]^{2}}{2}.
\]
\end{thm}
\begin{proof}
Define $t=n^{1-\kappa}.$ Then
\begin{align*}
\mathrm{{E}}\left[C_{\kappa}^{\lambda}\left(\frac{1}{n}\circ\sum_{j=1}^{n}X_{j}\right)\right] & =t\cdot\mathrm{{E}}\left[C_{\kappa}^{\lambda}\left(X\right)\right]\\
\mathrm{{E}}\left[C_{\kappa+1}^{\lambda}\left(\frac{1}{n}\circ\sum_{j=1}^{n}X_{j}\right)\right] & =t^{\frac{\kappa}{\kappa-1}}\cdot\mathrm{{E}}\left[C_{\kappa+1}^{\lambda}\left(X\right)\right].
\end{align*}
 Let $P_{t}^{\ast}$ denote the minimum information distribution satisfying
\[
\mathrm{{E}}_{P_{t}^{\ast}}\left[C_{\kappa}^{\lambda}\left(X\right)\right]=a\text{ and }\mathrm{{E}}_{P_{t}^{\ast}}\left[C_{\kappa+1}^{\lambda}\left(X\right)\right]=b
\]
 where 
\[
a=t\cdot\mathrm{{E}}\left[C_{\kappa}^{\lambda}\left(X\right)\right]\text{ and }b=t^{\frac{\kappa}{\kappa-1}}\cdot\mathrm{{E}}\left[C_{\kappa+1}^{\lambda}\left(X\right)\right].
\]
 Then $P_{t}^{\ast}$ is an element in the exponential family and
\[
\frac{P_{t}^{\ast}\left(j\right)}{\mathrm{{Po}}\left(\lambda,j\right)}=\frac{\exp\left(\beta_{1}C_{\kappa}^{\lambda}\left(j\right)+\beta_{2}C_{\kappa+1}^{\lambda}\left(j\right)\right)}{Z\left(\beta_{1},\beta_{2}\right)}
\]
 where $Z\left(\beta_{1},\beta_{2}\right)$ is the partition function
and $\beta_{1}$ and $\beta_{2}$ are determined by the conditions.
Thus 
\begin{align*}
 & D\left(P_{t}^{\ast}\Vert\mathrm{{Po}}\left(\lambda\right)\right)\\
\quad & =\sum_{x=0}^{\infty}P_{t}^{\ast}\left(x\right)\ln\frac{\exp\left(\beta_{1}C_{\kappa}^{\lambda}\left(x\right)+\beta_{2}C_{\kappa+1}^{\lambda}\left(x\right)\right)}{Z\left(\beta_{1},\beta_{2}\right)}\\
\quad & =\sum_{x=0}^{\infty}\left(\beta_{1}C_{\kappa}^{\lambda}\left(x\right)+\beta_{2}C_{\kappa+1}^{\lambda}\left(x\right)\right)P_{t}^{\ast}\left(x\right)-\ln Z\left(\beta_{1},\beta_{2}\right)\\
\quad & =\beta_{1}t\mathrm{{E}}\left[C_{\kappa}^{\lambda}\left(X\right)\right]+\beta_{2}t^{\frac{\kappa}{\kappa-1}}\cdot\mathrm{{E}}\left[C_{\kappa+1}^{\lambda}\left(X\right)\right]\\
 & \qquad\qquad\qquad\qquad\qquad\qquad\qquad\qquad\;\;-\ln Z\left(\beta_{1},\beta_{2}\right).
\end{align*}
Therefore 
\[
\frac{\mathrm{d}}{\mathrm{d}t}D\left(P_{t}^{\ast}\Vert\mathrm{{Po}}\left(\lambda\right)\right)=\left(\left.\begin{array}{c}
\frac{da}{dt}\\
\frac{db}{dt}
\end{array}\right\vert \begin{array}{c}
\frac{\partial D}{\partial a}\\
\frac{\partial D}{\partial b}
\end{array}\right)
\]
 where $\left(\cdot\mid\cdot\right)$ denotes the inner product. Thus
\begin{multline*}
\frac{\mathrm{d}^{2}}{\mathrm{d}t^{2}}D\left(P_{t}^{\ast}\Vert\mathrm{{Po}}\left(\lambda\right)\right)\\
=\left(\left.\begin{array}{c}
t\frac{d^{2}a}{dt^{2}}\\
t\frac{d^{2}b}{dt^{2}}
\end{array}\right\vert \begin{array}{c}
t^{-1}\frac{\partial D}{\partial a}\\
t^{-1}\frac{\partial D}{\partial b}
\end{array}\right)+\left(\begin{array}{c}
\frac{da}{dt}\\
\frac{db}{dt}
\end{array}\left\vert \begin{array}{cc}
\frac{\partial^{2}D}{\partial a^{2}} & \frac{\partial^{2}D}{\partial a\partial b}\\
\frac{\partial^{2}D}{\partial b\partial a} & \frac{\partial^{2}D}{\partial b^{2}}
\end{array}\right\vert \begin{array}{c}
\frac{da}{dt}\\
\frac{db}{dt}
\end{array}\right)\\
\rightarrow\mathrm{{E}}\left[C_{\kappa}^{\lambda}\left(X\right)\right]\frac{\partial^{2}D}{\partial a^{2}}\mathrm{{E}}\left[C_{\kappa}\left(X\right)\right]=\mathrm{{E}}\left[C_{\kappa}^{\lambda}\left(X\right)\right]^{2},
\end{multline*}
 where the physics notation $\left(\vec{u}\mid A\mid\vec{v}\right)$
for $\left(\vec{u}\mid A\vec{v}\right)$ is used when $A$ is a matrix.
Hence,
\begin{align*}
\liminf_{n\rightarrow\infty}n^{2\kappa-2}D & \left(\left.\frac{1}{n}\circ\sum_{j=1}^{n}X_{j}\right\Vert \mathrm{{Po}}\left(\lambda\right)\right)\\
 & \geq\liminf t^{-2}D\left(P_{t}^{\ast}\Vert\mathrm{{Po}}\left(\lambda\right)\right)\\
 & \geq\frac{\mathrm{{E}}\left[C_{\kappa}^{\lambda}\left(X\right)\right]^{2}}{2},
\end{align*}
 which proves the theorem. 
\end{proof}

\section{Discrete and continuous distributions\label{sec:disccont}}

The $\alpha$-thinning $\alpha\circ P$ of a distribution $P$ on
$\mathbb{N}_{0}$ is also a distribution on $\mathbb{N}_{0}$. We
can extend the thinning operation for distributions $P$ of random
variables $Y$ on $\mathbb{N}_{0}/n=\{0,\frac{1}{n},\frac{2}{n}\ldots\},$
by letting $\alpha\circ P$ be the distribution of $\frac{1}{n}\sum_{j=1}^{nY}B_{j}$,
where the $B_{j}$ are as before. More generally, starting with a
random variable $Y$ with distribution $P$ on $[0,\infty)$, let
$P_{n}$ denote the distribution on $\mathbb{N}_{0}/n$ with $P_{n}\left(\nicefrac{j}{n}\right)=P\left(\left[\nicefrac{j}{n},\nicefrac{j}{n}+1\right[\right).$
It is easy to see that $\alpha\circ P_{n}$ converges to the distribution
of $\alpha Y$ as $n\rightarrow\infty$. In this sense, thinning can
be interpreted as a discrete analog of the scaling operation for continuous
random variables.

Let $\Phi\left(\mu,\sigma^{2}\right)$ denote the distribution of
a Gaussian random variable with mean $\mu$ and variance $\sigma^{2}.$
We are interested in a lower bound on $D\left(X\Vert\Phi\left(\mu,\sigma^{2}\right)\right)$
in terms of the variance of $X,$ where $X$ is some random variable.
We shall assume that $\mathrm{{Var}}\left(X\right)\leq\sigma^{2}.$
First we remark that
\[
D\left(X\Vert\Phi\left(\mu,\sigma^{2}\right)\right)=D\left(aX+b\Vert\Phi\left(a\mu+b,a^{2}\sigma^{2}\right)\right)
\]
 for real constants $a$ and $b.$ The constants $a$ and $b$ can
be chosen so that $a\mu+b=a^{2}\sigma^{2}.$ Our next step is to discretize
\begin{multline}
D\left(aX+b\Vert\Phi\left(a\mu+b,a^{2}\sigma^{2}\right)\right)\\
\approx D\left(\left.\left\lfloor aX+b\right\rfloor \right\Vert \mathrm{{Po}}\left(a^{2}\sigma^{2}\right)\right).
\end{multline}
 Next we use Theorem \ref{nedregraense} to get 
\begin{multline*}
D\left(\left.\left\lfloor aX+b\right\rfloor \right\Vert \mathrm{{Po}}\left(a^{2}\sigma^{2}\right)\right)\geq\frac{\mathrm{{E}}\left[C_{2}^{a^{2}\sigma^{2}}\left(\left\lfloor aX+b\right\rfloor \right)\right]^{2}}{2}=\\
\frac{1}{4}\left(\frac{\mathrm{{Var}}\left\lfloor aX+b\right\rfloor }{a^{2}\sigma^{2}}-1\right)^{2}=\frac{1}{4}\left(\frac{\mathrm{{Var}}\left(\frac{\left\lfloor aX+b\right\rfloor }{a}\right)}{\sigma^{2}}-1\right)^{2}.
\end{multline*}
 Finally we use that $\mathrm{{Var}}\left(\left\lfloor aX+b\right\rfloor /a\right)\rightarrow\mathrm{{Var}}\left(X\right)$
for $n\rightarrow\infty$ to get
\begin{align*}
D\left(X\Vert\Phi\left(\mu,\sigma^{2}\right)\right) & \geq\frac{\left(\frac{\mathrm{{Var}}\left(X\right)}{\sigma^{2}}-1\right)^{2}}{4}\\
 & =\frac{\mathrm{{E}}\left[H_{2}\left(X-\mathrm{{E}}\left[X\right]\right)\right]^{2}}{2},
\end{align*}
where $H_{2}$ is the second Hermite polynomial. This inequality can
also be proved by a straightforward calculation in the exponential
family of Gaussian distributions. Following the same kind of reasoning
we get the following new and non-trivial result:
\begin{thm}
\label{Thm14}For any random variable $X$ with mean $0$ and variance
$1$ and for any $\ell\in\mathbb{N}$ there exists $\varepsilon>0$
such 
\[
D\left(X\Vert\Phi\left(\mu,\sigma^{2}\right)\right)\geq\frac{\mathrm{{E}}\left[H_{2\ell}\left(X\right)\right]^{2}}{2}
\]
 if $\mathrm{{E}}\left[H_{2\ell}\left(X\right)\right]\in\left[-\varepsilon,0\right].$ \end{thm}
\begin{proof}
We have to prove that $\mathrm{{E}}\left[\left(H_{2\ell}\left(X\right)\right)^{3}\right]>0.$
The product of two Hermite polynomials can be expanded as
\[
H_{m}\left(x\right)H_{n}\left(x\right)=\sum_{r=0}^{\min\left(m,n\right)}r!\binom{m}{r}\binom{n}{r}H_{m+n-2r}\left(x\right).
\]
This result was proved by Feldheim and later extended \cite{Feldheim1940,Carlitz1962},
but can also be derived by the result of Khokhlov by approximating
a Gaussian distribution by a Poisson distribution. Using this result
we get
\[
\left(H_{2\ell}\left(X\right)\right)^{2}=\sum_{r=0}^{2\ell}r!\binom{2\ell}{r}^{2}H_{4\ell-2r}\left(x\right)
\]
 and
\[
\mathrm{{E}}\left[\left(H_{2\ell}\left(X\right)\right)^{3}\right]=\ell!\binom{2\ell}{\ell}^{2}\geq0.
\]

\end{proof}
Note that the Hermite polynomials of even orders are essentially generalized
Laguerre polynomials so the theorem can be translated into a result
on these polynomials.

If Conjecture \ref{KonjMindre} holds then the condition $\mathrm{{E}}\left[H_{2\ell}\left(X\right)\right]\in\left[-\varepsilon,0\right]$
in Theorem \ref{Thm14} can be replaced by the much weaker condition
$\mathrm{{E}}\left[H_{2\ell}\left(X\right)\right]\leq0.$ The case
$\ell=1$ has been discussed above and the case $\ell=2$ has also
been proved \cite{Harremoes2006b}%
\footnote{Note that \cite{Harremoes2006b} used a different normalization of
the Hermite polynomials.%
}. Assume that $\mathrm{{E}}\left[H_{2\ell}\left(X_{1}\right)\right]$
is the first Hermite moment that is different from zero and that it
is negative. Let $X_{1},X_{2},\dots$ be a sequence of iid random
variables and let $U_{n}$ denote their normalized sum. In \cite{Harremoes2006b}
it was proved that the rate of convergence in the information theoretic
version of the Central Limit Theorem satisfies
\[
\liminf_{n\to\infty}n^{2\ell-2}D\left(U_{n}\Vert\Phi\left(0,1\right)\right)\geq\frac{\left(\mathrm{{E}}\left[H_{2\ell}\left(X_{1}\right)\right]\right)^{2}}{2}
\]
 Theorem \ref{Thm14} implies the stronger result that 
\begin{equation}
n^{2\ell-2}D\left(U_{n}\Vert\Phi\left(0,1\right)\right)\geq\frac{\left(\mathrm{{E}}\left[H_{2\ell}\left(X_{1}\right)\right]\right)^{2}}{2}\label{nedreasymp}
\end{equation}
 holds eventually. We conjecture that the Inequality (\ref{nedreasymp})
holds for all $n,$ and that this lower bound will give the correct
rate of convergence in the information theoretic version of the Central
Limit Theorem under weak regularity conditions .

\section{Acknowledgment}

Lemma \ref{PeterG} was developed in close collaboration with Peter
Gr{\"u}nwald.

\section{Appendix}

The following result is a multinomial version of \emph{Vandermonde's
identity} and is easily proved by induction.
\begin{lem}
The falling factorial satisfies the multinomial expansion, i.e.
\[
\left(\sum_{x=1}^{Y}X_{x}\right)^{\underline{k}}=\sum_{\sum k_{x}=k}\binom{k}{\begin{array}{cccc}
k_{1} & k_{2} & \cdots & k_{Y}\end{array}}{\displaystyle \prod\limits _{x=1}^{Y}}X_{x}^{\underline{k}_{x}}~.
\]

\end{lem}
We apply this result on the thinning operation.
\begin{lem}
\label{thm:mombeh} Let $X_{1},X_{2},...$ be a sequence of independent
identically distributed discrete random variables all distributed
like $X.$ Then
\begin{multline*}
\mathrm{{E}}\left[\left(\frac{1}{n}\circ\sum_{j=1}^{n}X_{j}\right)^{\underline{k}}\right]=\\
\left\{ \begin{array}{ll}
\begin{array}{c}
\lambda^{k}\ ,\end{array} & \text{for }k\leq\kappa-1;\\
\begin{array}{c}
\lambda^{k}+\frac{\mathrm{{E}}\left[X^{\underline{\kappa}}\right]-\lambda^{k}}{n^{\kappa-1}}\ ,\end{array}\bigskip & \text{for }k=\kappa~;\\
\begin{array}{c}
\lambda^{\kappa+1}+\frac{\left(n-1\right)\left(\kappa+1\right)\lambda\left(\mathrm{{E}}\left[X^{\underline{\kappa}}\right]-\lambda^{\kappa}\right)}{2n^{\kappa}}\smallskip\\
+\frac{\mathrm{{E}}\left[X^{\underline{\kappa+1}}\right]-\lambda^{\kappa+1}}{n^{\kappa}},
\end{array} & \text{for }k=\kappa+1.
\end{array}\right.
\end{multline*}
\end{lem}
\begin{proof}
These equations follow from the Vandermonde Identity for factorials
combined with Equation (\ref{Factorial}). 
\end{proof}

\begin{proof}[Proof of Lemma \ref{Lemma:kvad}]
 For any $n$ satisfying $0\leq n\leq k$ we can use the Vandermonde
Identity for factorials to get
\begin{multline*}
\sum_{\mu=0}^{k}\sum_{\nu=0}^{k}\binom{k}{\mu}\binom{k}{\nu}\mu^{\underline{n}}\nu^{\underline{n}}\left(\mu+\nu-n\right)^{\underline{k}}\left(\textrm{-}1\right)^{\mu+\nu}=\\
\sum_{\substack{a+b\\
+c=k
}
}\binom{k}{a\, b\, c}n^{\underline{c}}\left(\begin{array}{c}
\sum_{\mu=0}^{k}\binom{k}{\mu}\mu^{\underline{n}}\left(\mu-n\right)^{\underline{a}}\left(\textrm{-}1\right)^{\mu}\\
\times\\
\sum_{\nu=0}^{k}\binom{k}{\nu}\nu^{\underline{n}}\left(\nu-n\right)^{\underline{b}}\left(\textrm{-}1\right)^{\nu}
\end{array}\right).
\end{multline*}
 Now 
\begin{align*}
\sum_{\mu=0}^{k} & \binom{k}{\mu}\mu^{\underline{n}}\left(\mu-n\right)^{\underline{a}}\left(\textrm{-}1\right)^{\mu}\\
 & =\sum_{\mu=n+a}^{k}\frac{k^{\underline{n+a}}\left(k-n-a\right)!}{\left(\mu-n-a\right)!\left(k-\mu\right)!}\left(\textrm{-}1\right)^{\mu}\\
 & =k^{\underline{n+a}}\left(\textrm{-}1\right)^{n+a}\sum_{\mu=n+a}^{k}\binom{k-n-a}{\mu-n-a}\left(\textrm{-}1\right)^{\mu-n-a}\\
 & =\left\{ \begin{array}{ll}
0, & \text{for }a\neq k-n;\\
k!\left(\textrm{-}1\right)^{k}, & \text{for }a=k-n.
\end{array}\right.
\end{align*}
 Thus, this sum is only non-zero if $a=k-n.$ Similarly the sum $\sum_{\nu=0}^{k}\binom{k}{\nu}\nu^{\underline{n}}\left(\textrm{-}1\right)^{\nu}\left(\nu-n\right)^{\underline{b}}$
is only non-zero if $b=k-n.$ Thus $c=k-2\left(k-n\right)=2n-k$ and
the condition $c\geq0$ implies that $n\geq\nicefrac{k}{2}.$ Hence
\begin{align*}
 & \sum_{\mu=0}^{k}\sum_{\nu=0}^{k}\binom{k}{\mu}\binom{k}{\nu}\mu^{\underline{n}}\nu^{\underline{n}}\left(\mu+\nu-n\right)^{\underline{k}}\left(-1\right)^{\mu+\nu}\\
 & \quad=\binom{k}{\begin{array}{ccc}
k-n & k-n & 2n-k\end{array}}n^{\underline{2n-k}}k!\left(-1\right)^{k}k!\left(\textrm{-}1\right)^{k}\\
 & \quad=\left(k^{\underline{n}}\right)^{3}n^{\underline{k-n}},
\end{align*}
 which is always non-negative and it is positive if $\nicefrac{k}{2}\leq n\leq k.$ 
\end{proof}

\begin{proof}[Proof of Theorem \ref{nedregraense} Case 1]
 The theorem holds if $\lambda\in\mathbb{N},$ and from now on we
will assume that $\lambda\not\in\mathbb{N}.$ In the interval $\left]\lambda,\lambda+1\right]$
the ceiling $\left\lceil \lambda\right\rceil $ is the only integer
and $C_{2}^{\lambda}\left(x\right)$ is minimal for $x=\left\lceil \lambda\right\rceil .$
Therefore
\begin{multline*}
\frac{\mathrm{d}^{2}Z}{\mathrm{d}\beta^{2}}=\sum_{x=\left\lfloor \lambda\right\rfloor ,\left\lceil \lambda\right\rceil }\mathrm{Po}\left(\lambda,x\right)f\left(C_{2}^{\lambda}\left(x\right)\right)\\
+\sum_{x\neq\left\lfloor \lambda\right\rfloor ,\left\lceil \lambda\right\rceil }\mathrm{Po}\left(\lambda,x\right)f\left(C_{2}^{\lambda}\left(x\right)\right)\\
\leq\sum_{x=\left\lfloor \lambda\right\rfloor ,\left\lceil \lambda\right\rceil }\mathrm{Po}\left(\lambda,x\right)f\left(C_{2}^{\lambda}\left(x\right)\right)+\sum_{x\neq\left\lfloor \lambda\right\rfloor ,\left\lceil \lambda\right\rceil }\mathrm{Po}\left(\lambda,x\right)
\end{multline*}
 so we just have to show that
\[
\sum_{x=\left\lfloor \lambda\right\rfloor ,\left\lceil \lambda\right\rceil }\mathrm{Po}\left(\lambda,x\right)f\left(C_{2}^{\lambda}\left(x\right)\right)\leq\sum_{x=\left\lfloor \lambda\right\rfloor ,\left\lceil \lambda\right\rceil }\mathrm{Po}\left(\lambda,x\right).
\]
After division by $\mathrm{Po}\left(\lambda,\left\lfloor \lambda\right\rfloor \right)$
and $1+\frac{\lambda}{\left\lceil \lambda\right\rceil }$ we see that
it is sufficient to prove that
\[
\frac{\left\lceil \lambda\right\rceil f\left(C_{2}^{\lambda}\left(\left\lfloor \lambda\right\rfloor \right)\right)+\lambda f\left(C_{2}^{\lambda}\left(\left\lceil \lambda\right\rceil \right)\right)}{\left\lceil \lambda\right\rceil +\lambda}\leq1.
\]
At this point the proof splits into the three sub-cases: $\lambda\in\left[0,1\right],\ \lambda\in\left]1,2\right],$
and $\lambda\in\left]2,\infty\right[.$

\textbf{Sub-case }$\lambda<1$ For $\lambda\in\left]0,1\right[$ the
ceiling of $\lambda$ is $1$ and $C_{2}^{\lambda}\left(0\right)=\frac{\lambda}{2^{\nicefrac{1}{2}}}$
and $C_{2}^{\lambda}\left(1\right)=\frac{\lambda-2}{2^{\nicefrac{1}{2}}}.$
We have 
\begin{align*}
 & \frac{f\left(\frac{\lambda}{2^{\nicefrac{1}{2}}}\right)+\lambda f\left(\frac{\lambda-2}{2^{\nicefrac{1}{2}}}\right)}{1+\lambda}\\
 & \qquad=\frac{\lambda^{2}\exp\left(\beta_{0}\frac{\lambda}{2^{\nicefrac{1}{2}}}\right)+\left(\lambda-2\right)^{2}\lambda\exp\left(\beta_{0}\frac{\lambda-2}{2^{1/2}}\right)}{2\left(1+\lambda\right)}\\
 & \qquad=\frac{\lambda^{2}+\left(\lambda-2\right)^{2}\lambda\exp\left(\textrm{-}\beta_{0}2^{1/2}\right)}{2\left(1+\lambda\right)\left(1-\beta_{0}\frac{\lambda}{2^{\nicefrac{1}{2}}}\right)}.
\end{align*}
 As a function of $\lambda\in\left[0,1\right]$ this function is increasing
for $\lambda\leq0.458471$ and decreasing for $\lambda\geq0.455\,6$
and the maximal value is $0.928\,8$, which is less than 1. 

\textbf{Sub-case }$1<\lambda<2$ For $\lambda\in\left]1,2\right[$
the ceiling is 2, and we have
\begin{gather*}
\frac{\left\lceil \lambda\right\rceil f\left(C_{2}^{\lambda}\left(\left\lfloor \lambda\right\rfloor \right)\right)+\lambda f\left(C_{2}^{\lambda}\left(\left\lceil \lambda\right\rceil \right)\right)}{\left\lceil \lambda\right\rceil +\lambda}=\\
\frac{2f\left(C_{2}^{\lambda}\left(1\right)\right)+\lambda f\left(C_{2}^{\lambda}\left(2\right)\right)}{2+\lambda}=\\
\frac{\left(\lambda-2\right)^{2}\exp\left(\beta_{0}\frac{\lambda-2}{2^{\nicefrac{1}{2}}}\right)+\frac{\left(\lambda^{2}-4\lambda+2\right)^{2}}{2\lambda}\exp\left(\beta_{0}\frac{\lambda^{2}-4\lambda+2}{2^{\nicefrac{1}{2}}\lambda}\right)}{2+\lambda}\leq\\
\frac{\left(\lambda-2\right)^{2}\exp\left(\textrm{-}\beta_{0}2^{\nicefrac{1}{2}}\right)+\frac{\left(\lambda^{2}-4\lambda+2\right)^{2}}{2\lambda}\exp\left(\beta_{0}\left(2-2^{\nicefrac{3}{2}}\right)\right)}{2+\lambda}.
\end{gather*}
 For $\lambda\in\left[1,2\right]$ this function is decreasing in
$\lambda,$ with maximum for $\lambda=1$ attaining the value $0.878\,4,$
which is less than 1.

\textbf{Sub-case }$\lambda>2$ For $\lambda>2$ we use that the Poisson
distribution has mode $\left\lfloor \lambda\right\rfloor .$ The minimum
of $C_{2}^{\lambda}\left(X\right)$ as a function of $X$ is $C_{2}^{\lambda}\left(\left\lceil \lambda\right\rceil \right).$
The minimum of the function $\lambda\rightarrow C_{2}^{\lambda}\left(\left\lceil \lambda\right\rceil \right)$
is attained when $\left\lceil \lambda\right\rceil =3.$ We find the
minimum of $C_{2}^{\lambda}\left(3\right)$ to be $\textrm{-}0.7785$
as a function with domain $\mathbb{R}_{+}$ has minimum $\textrm{-}\frac{1+\left(4\lambda\right)^{\textrm{-}1}}{2^{\nicefrac{1}{2}}}$
for $x=\lambda+\nicefrac{1}{2},$ to get 
\begin{multline}
\frac{\left\lceil \lambda\right\rceil f\left(C_{2}^{\lambda}\left(\left\lfloor \lambda\right\rfloor \right)\right)+\lambda f\left(C_{2}^{\lambda}\left(\left\lceil \lambda\right\rceil \right)\right)}{\left\lceil \lambda\right\rceil +\lambda}\\
\leq\frac{1}{2}f\left(\textrm{-}2^{\textrm{-}\nicefrac{1}{2}}\right)+\frac{1}{2}f\left(\textrm{-}0.7785\right)=.9704<1,
\end{multline}
 which proves the theorem in this case.
\end{proof}

\begin{proof}[Proof of Theorem \ref{nedregraense} Case 2]
 This case deals with the situation where $\mathrm{{E}}\left[C_{2}^{\lambda}\left(X\right)\right]<\beta_{0},$
which is only possible if $\lambda\leq4.$

The function $C_{2}^{\lambda}$ satisfies $C_{2}^{\lambda}\left(\lambda\right)=C_{2}^{\lambda}\left(\lambda+1\right)=\textrm{-}2^{\textrm{-}\nicefrac{1}{2}}.$
Convexity of the function $C_{2}^{\lambda}$ and the fact that $X$
can only assume integer values implies that if $\mathrm{{E}}\left[C_{2}^{\lambda}\left(X\right)\right]\leq\textrm{-}2^{\textrm{-}\nicefrac{1}{2}}$
then $\Pr\left(X=\left\lceil \lambda\right\rceil \right)\geq\nicefrac{1}{2}.$

\textbf{Sub-case }$1\leq\lambda\leq4$ For $\lambda\in\left[1,4\right]$
we can get an even better bound on $P\left(X=\left\lceil \lambda\right\rceil \right)$
as follows. 
\begin{multline*}
\mathrm{E}\left[C_{2}^{\lambda}\left(X\right)\right]\geq\Pr\left(X=\left\lceil \lambda\right\rceil \right)C_{2}^{\lambda}\left(\left\lceil \lambda\right\rceil \right)\\
+\left(1-\Pr\left(X=\left\lceil \lambda\right\rceil \right)\right)\min\left\{ C_{2}^{\lambda}\left(\left\lceil \lambda\right\rceil -1\right),C_{2}^{\lambda}\left(\left\lceil \lambda\right\rceil +1\right)\right\} .
\end{multline*}
 Hence $\mathrm{E}\left[C_{2}^{\lambda}\left(X\right)\right]$ can
be lower bounded by an expression of the form 
\begin{equation}
\qquad\left(\frac{1}{2}+t\right)C_{2}^{\lambda}\left(x\right)+\left(\frac{1}{2}-t\right)C_{2}^{\lambda}\left(x+1\right),\label{Eq:Mix}
\end{equation}
 where $s=\Pr\left(X=\left\lceil \lambda\right\rceil \right)$ or
$s=1-\Pr\left(X=\left\lceil \lambda\right\rceil \right)$ and $x=\left\lceil \lambda\right\rceil $
or $x=\left\lceil \lambda\right\rceil -1.$ If $x$ is allowed to
assume real values in (\ref{Eq:Mix}) then it is a polynomial in $x$
of order 2 with minimum
\[
\textrm{-}\frac{1}{2^{\nicefrac{1}{2}}}-\frac{t^{2}}{2^{\nicefrac{1}{2}}\lambda}.
\]
 Hence the condition $\mathrm{{E}}\left[C_{2}^{\lambda}\left(X\right)\right]<\beta_{0}$
implies that
\[
\textrm{-}\frac{1}{2^{\nicefrac{1}{2}}}-\frac{t^{2}}{2^{\nicefrac{1}{2}}\lambda}<\beta_{0}
\]
 and
\[
t>\left(\lambda\left(\textrm{-}\beta_{0}2^{\nicefrac{1}{2}}-1\right)\right)^{\nicefrac{1}{2}}.
\]
 Hence 
\[
\Pr\left(X=\left\lceil \lambda\right\rceil \right)\geq\frac{1}{2}+t>\frac{1}{2}+\left(\lambda\left(\textrm{-}\beta_{0}2^{\nicefrac{1}{2}}-1\right)\right)^{\nicefrac{1}{2}}.
\]
 The maximal point probability of a Poisson random variable with $\lambda\geq1$
is $\exp\left(\textrm{-}1\right)$ implying that
\begin{align*}
 & \left\Vert P-\mathrm{Po}\left(\lambda\right)\right\Vert \\
 & \qquad\geq2\left(\Pr\left(X=\left\lceil \lambda\right\rceil \right)-\mathrm{Po}\left(\lambda,\left\lceil \lambda\right\rceil \right)\right)\\
 & \qquad\geq2\left(\frac{1}{2}+\left(\lambda\left(\textrm{-}\beta_{0}2^{\nicefrac{1}{2}}-1\right)\right)^{\nicefrac{1}{2}}-\exp\left(\textrm{-}1\right)\right).
\end{align*}
 Therefore by (\ref{eqn:total}) it is sufficient to prove that
\begin{equation}
2\left(\frac{1}{2}+\left(\lambda\left(\textrm{-}\beta_{0}2^{\nicefrac{1}{2}}-1\right)\right)^{\nicefrac{1}{2}}-\exp\left(\textrm{-}1\right)\right)\geq\left\vert C_{2}^{\lambda}\left(\left\lceil \lambda\right\rceil \right)\right\vert .\label{Ineq14}
\end{equation}
This inequality is checked numerically and illustrated in Figure \ref{fig:14}.

\begin{figure}
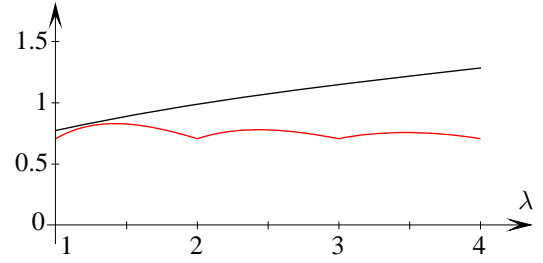

\begin{centering}
\begin{pgfpicture}{31.41mm}{34.11mm}{104.50mm}{70.31mm} 
\pgfsetxvec{\pgfpoint{1.00mm}{0mm}} 
\pgfsetyvec{\pgfpoint{0mm}{1.00mm}} 
\color[rgb]{0,0,0}\pgfsetlinewidth{0.30mm}\pgfsetdash{}{0mm} 
\pgfsetlinewidth{0.12mm}\pgfmoveto{\pgfxy(36.36,36.37)}\pgflineto{\pgfxy(36.36,68.31)}\pgfstroke \pgfmoveto{\pgfxy(36.36,68.31)}\pgflineto{\pgfxy(35.66,65.51)}\pgflineto{\pgfxy(36.36,66.91)}\pgflineto{\pgfxy(37.06,65.51)}\pgflineto{\pgfxy(36.36,68.31)}\pgfclosepath\pgffill \pgfmoveto{\pgfxy(36.36,68.31)}\pgflineto{\pgfxy(35.66,65.51)}\pgflineto{\pgfxy(36.36,66.91)}\pgflineto{\pgfxy(37.06,65.51)}\pgflineto{\pgfxy(36.36,68.31)}\pgfclosepath\pgfstroke 
\pgfmoveto{\pgfxy(35.22,38.92)}\pgflineto{\pgfxy(99.58,38.92)}\pgfstroke
\pgfmoveto{\pgfxy(99.58,38.92)}\pgflineto{\pgfxy(96.78,39.62)}\pgflineto{\pgfxy(98.18,38.92)}\pgflineto{\pgfxy(96.78,38.22)}\pgflineto{\pgfxy(99.58,38.92)}\pgfclosepath\pgffill \pgfmoveto{\pgfxy(99.58,38.92)}\pgflineto{\pgfxy(96.78,39.62)}\pgflineto{\pgfxy(98.18,38.92)}\pgflineto{\pgfxy(96.78,38.22)}\pgflineto{\pgfxy(99.58,38.92)}\pgfclosepath\pgfstroke  
\pgfmoveto{\pgfxy(92.87,39.38)}\pgflineto{\pgfxy(92.87,38.33)}\pgfstroke 
\pgfmoveto{\pgfxy(83.46,39.38)}\pgflineto{\pgfxy(83.46,38.33)}\pgfstroke  
\pgfmoveto{\pgfxy(74.05,39.38)}\pgflineto{\pgfxy(74.05,38.33)}\pgfstroke 
\pgfmoveto{\pgfxy(64.64,39.38)}\pgflineto{\pgfxy(64.64,38.33)}\pgfstroke
\pgfputat{\pgfxy(36.98,35.00)}{\pgfbox[bottom,left]{1}}
\pgfputat{\pgfxy(54.23,35.00)}{\pgfbox[bottom,left]{2}}
\pgfputat{\pgfxy(73.05,35.00)}{\pgfbox[bottom,left]{3}}
\pgfputat{\pgfxy(91.87,35.00)}{\pgfbox[bottom,left]{4}}
\pgfmoveto{\pgfxy(55.23,39.38)}\pgflineto{\pgfxy(55.23,38.33)}\pgfstroke
\pgfmoveto{\pgfxy(45.83,39.38)}\pgflineto{\pgfxy(45.83,38.33)}\pgfstroke
\pgfputat{\pgfxy(33.24,38.21)}{\pgfbox[bottom,left]{0}}
\pgfputat{\pgfxy(31.00,46.08)}{\pgfbox[bottom,left]{0.5}}
\pgfputat{\pgfxy(33.67,54.24)}{\pgfbox[bottom,left]{1}}
\pgfputat{\pgfxy(31.00,62.39)}{\pgfbox[bottom,left]{1.5}} 
\pgfmoveto{\pgfxy(36.88,63.32)}\pgflineto{\pgfxy(35.84,63.32)}\pgfstroke 
\pgfmoveto{\pgfxy(36.88,55.16)}\pgflineto{\pgfxy(35.84,55.16)}\pgfstroke  
\pgfmoveto{\pgfxy(36.88,47.01)}\pgflineto{\pgfxy(35.84,47.01)}\pgfstroke 
\color[rgb]{1,0,0}\pgfsetlinewidth{0.18mm}\pgfmoveto{\pgfxy(55.28,50.39)}\pgfcurveto{\pgfxy(55.23,50.41)}{\pgfxy(54.90,50.49)}{\pgfxy(54.86,50.50)}\pgfcurveto{\pgfxy(48.96,52.25)}{\pgfxy(42.03,53.75)}{\pgfxy(36.41,50.39)}\pgfstroke
\pgfmoveto{\pgfxy(55.23,50.39)}\pgfcurveto{\pgfxy(61.29,52.39)}{\pgfxy(67.93,51.53)}{\pgfxy(74.05,50.39)}\pgfstroke
\pgfmoveto{\pgfxy(74.05,50.39)}\pgfcurveto{\pgfxy(76.79,50.99)}{\pgfxy(79.62,51.14)}{\pgfxy(82.42,51.22)}\pgfcurveto{\pgfxy(85.91,51.22)}{\pgfxy(89.42,50.96)}{\pgfxy(92.87,50.39)}\pgfstroke 
\color[rgb]{0,0,0}\pgfmoveto{\pgfxy(36.41,51.48)}\pgfcurveto{\pgfxy(55.03,55.54)}{\pgfxy(74.04,57.64)}{\pgfxy(92.95,59.81)}\pgfstroke 
\pgfputat{\pgfxy(97.89,41.01)}{\pgfbox[bottom,left]{$\lambda$}} 
\end{pgfpicture}
\par\end{centering}

\protect\caption{\label{fig:14}  The left and the right hand side of the Inequality
(\ref{Ineq14}).}

\end{figure}

\textbf{Sub-case }$\mathbf{0}\leq\lambda<\nicefrac{1}{2}$ For $\lambda\in\left[0,\nicefrac{1}{2}\right]$
we have 
\begin{align*}
\mathrm{{E}}\left[C_{2}^{\lambda}\left(X\right)\right] & \geq\Pr\left(1\right)C_{2}^{\lambda}\left(1\right)+\left(1-\Pr\left(1\right)\right)C_{2}^{\lambda}\left(0\right)\\
 & =\Pr\left(1\right)\frac{\textrm{-}2+\lambda}{2^{\nicefrac{1}{2}}}+\left(1-\Pr\left(1\right)\right)\frac{\lambda}{2^{\nicefrac{1}{2}}}\\
 & =\frac{\lambda-2\Pr\left(1\right)}{2^{\nicefrac{1}{2}}}.
\end{align*}
 Therefore by (\ref{eqn:total}) it is sufficient to prove that
\[
\frac{2\Pr\left(1\right)-\lambda}{2^{\nicefrac{1}{2}}}\leq2\left(\Pr\left(1\right)-\lambda\exp\left(\textrm{-}\lambda\right)\right),
\]
 which is equivalent to
\[
2^{\nicefrac{1}{2}}\lambda\exp\left(\textrm{-}\lambda\right)-\frac{\lambda}{2}\leq\left(2^{\nicefrac{1}{2}}-1\right)\Pr\left(1\right).
\]
 We know that $P\left(1\right)\geq\nicefrac{1}{2}$ so it is sufficient
to prove that 
\begin{equation}
2^{\nicefrac{1}{2}}\lambda\exp\left(\textrm{-}\lambda\right)-\frac{\lambda}{2}\leq\frac{2^{\nicefrac{1}{2}}-1}{2}.\label{Ineq0halv}
\end{equation}
 This inequality is checked numerically and illustrated in Figure
\ref{fig:lambda0halv}.

\begin{figure}
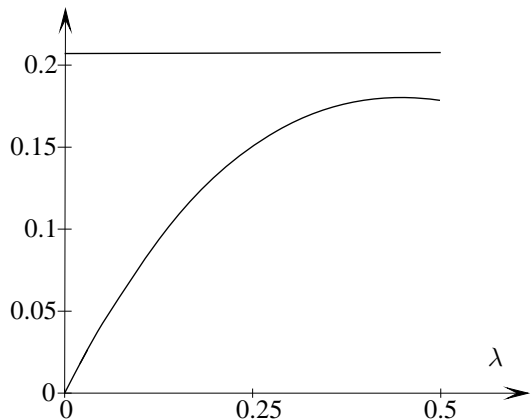

\begin{centering}
\begin{pgfpicture}{3.49mm}{14.00mm}{90.15mm}{72.55mm} 
\pgfsetxvec{\pgfpoint{1.00mm}{0mm}} 
\pgfsetyvec{\pgfpoint{0mm}{1.00mm}} 
\color[rgb]{0,0,0}\pgfsetlinewidth{0.30mm}\pgfsetdash{}{0mm} 
\pgfputat{\pgfxy(5.49,65.72)}{\pgfbox[bottom,left]{\fontsize{0.16}{0.19}\selectfont }} 
\pgfsetlinewidth{0.12mm}\pgfmoveto{\pgfxy(20.85,19.40)}\pgflineto{\pgfxy(20.96,70.55)}\pgfstroke \pgfmoveto{\pgfxy(20.96,70.55)}\pgflineto{\pgfxy(20.25,67.75)}\pgflineto{\pgfxy(20.95,69.15)}\pgflineto{\pgfxy(21.65,67.75)}\pgflineto{\pgfxy(20.96,70.55)}\pgfclosepath\pgffill \pgfmoveto{\pgfxy(20.96,70.55)}\pgflineto{\pgfxy(20.25,67.75)}\pgflineto{\pgfxy(20.95,69.15)}\pgflineto{\pgfxy(21.65,67.75)}\pgflineto{\pgfxy(20.96,70.55)}\pgfclosepath\pgfstroke 
\pgfmoveto{\pgfxy(20.76,19.45)}\pgflineto{\pgfxy(82.47,19.44)}\pgfstroke
\pgfmoveto{\pgfxy(82.47,19.44)}\pgflineto{\pgfxy(79.67,20.14)}\pgflineto{\pgfxy(81.07,19.44)}\pgflineto{\pgfxy(79.67,18.74)}\pgflineto{\pgfxy(82.47,19.44)}\pgfclosepath\pgffill \pgfmoveto{\pgfxy(82.47,19.44)}\pgflineto{\pgfxy(79.67,20.14)}\pgflineto{\pgfxy(81.07,19.44)}\pgflineto{\pgfxy(79.67,18.74)}\pgflineto{\pgfxy(82.47,19.44)}\pgfclosepath\pgfstroke 
\pgfmoveto{\pgfxy(70.85,20.06)}\pgflineto{\pgfxy(70.85,18.83)}\pgfstroke
\pgfputat{\pgfxy(20.20,16.10)}{\pgfbox[bottom,left]{0}}
\pgfputat{\pgfxy(43.64,16.10)}{\pgfbox[bottom,left]{0.25}}
\pgfputat{\pgfxy(68.65,16.10)}{\pgfbox[bottom,left]{0.5}}
\pgfmoveto{\pgfxy(45.84,20.06)}\pgflineto{\pgfxy(45.84,18.83)}\pgfstroke  
\pgfmoveto{\pgfxy(20.85,20.06)}\pgflineto{\pgfxy(20.85,18.83)}\pgfstroke
\pgfmoveto{\pgfxy(21.78,63.09)}\pgflineto{\pgfxy(19.92,63.09)}\pgfstroke 
\pgfputat{\pgfxy(17.85,18.35)}{\pgfbox[bottom,left]{0}}
\pgfputat{\pgfxy(13.65,29.25)}{\pgfbox[bottom,left]{0.05}}
\pgfputat{\pgfxy(15.50,40.17)}{\pgfbox[bottom,left]{0.1}}
\pgfputat{\pgfxy(13.65,51.07)}{\pgfbox[bottom,left]{0.15}} 
\pgfputat{\pgfxy(15.50,61.99)}{\pgfbox[bottom,left]{0.2}}
\pgfmoveto{\pgfxy(21.78,52.17)}\pgflineto{\pgfxy(19.92,52.17)}\pgfstroke 
\pgfmoveto{\pgfxy(21.78,41.27)}\pgflineto{\pgfxy(19.92,41.27)}\pgfstroke 
\pgfmoveto{\pgfxy(21.78,30.35)}\pgflineto{\pgfxy(19.92,30.35)}\pgfstroke 
\pgfmoveto{\pgfxy(21.78,19.45)}\pgflineto{\pgfxy(19.92,19.45)}\pgfstroke 
\pgfsetlinewidth{0.18mm}\pgfmoveto{\pgfxy(22.87,23.34)}\pgflineto{\pgfxy(23.89,25.21)}\pgfstroke 
\pgfmoveto{\pgfxy(20.85,19.45)}\pgfcurveto{\pgfxy(22.46,22.56)}{\pgfxy(24.04,25.70)}{\pgfxy(25.90,28.75)}\pgfcurveto{\pgfxy(31.78,37.81)}{\pgfxy(37.60,47.74)}{\pgfxy(49.13,54.37)}\pgfcurveto{\pgfxy(56.25,58.45)}{\pgfxy(63.65,59.43)}{\pgfxy(70.75,58.40)}\pgfstroke \pgfmoveto{\pgfxy(70.85,64.76)}\pgflineto{\pgfxy(20.85,64.63)}\pgfstroke 
\pgfputat{\pgfxy(77.11,23.04)}{\pgfbox[bottom,left]{$\lambda$}} 
\end{pgfpicture}% 
\par\end{centering}

\protect\caption{\label{fig:lambda0halv}Plot of the function $2^{\nicefrac{1}{2}}\lambda\exp\left(-\lambda\right)-\nicefrac{\lambda}{2}$
and the constant function $\frac{2^{\nicefrac{1}{2}}-1}{2}$ .}
\end{figure}

\textbf{Sub-case }$\nicefrac{1}{2}\leq\lambda<1$ For $\lambda\in\left[\nicefrac{1}{2},1\right]$
we have 
\begin{align*}
\mathrm{{E}}\left[C_{2}^{\lambda}\left(X\right)\right] & \geq\Pr\left(1\right)C_{2}^{\lambda}\left(1\right)+\left(1-\Pr\left(1\right)\right)C_{2}^{\lambda}\left(2\right)\\
 & =\Pr\left(1\right)\frac{-2+\lambda}{2^{\nicefrac{1}{2}}}+\left(1-\Pr\left(1\right)\right)\frac{\lambda^{2}-4\lambda+2}{2^{\nicefrac{1}{2}}\lambda}\\
 & =\frac{\lambda^{2}-4\lambda+2}{2^{\nicefrac{1}{2}}\lambda}+\Pr\left(1\right)\frac{2\lambda-2}{2^{\nicefrac{1}{2}}\lambda}.
\end{align*}
 The inequality $\mathrm{{E}}\left[C_{2}^{\lambda}\left(X\right)\right]\leq\beta_{0}$
implies that
\[
\frac{\lambda^{2}-4\lambda+2}{2^{\nicefrac{1}{2}}\lambda}+\Pr\left(1\right)\frac{2\lambda-2}{2^{\nicefrac{1}{2}}\lambda}\leq\beta_{0}
\]
 and
\[
\Pr\left(1\right)\geq\frac{\beta_{0}2^{\nicefrac{1}{2}}\lambda-\lambda^{2}+4\lambda-2}{2\lambda-2}\,.
\]
 Therefore by (\ref{eqn:total}) it is sufficient to prove that
\[
\frac{\lambda^{2}-4\lambda+2}{2^{\nicefrac{1}{2}}\lambda}+\Pr\left(1\right)\frac{2\lambda-2}{2^{\nicefrac{1}{2}}\lambda}\leq2\left(\Pr\left(1\right)-\lambda\exp\left(\textrm{-}\lambda\right)\right),
\]
 which is equivalent to
\[
\frac{\lambda^{2}-4\lambda+2}{2^{\nicefrac{1}{2}}\lambda}+2\lambda\exp\left(\textrm{-}\lambda\right)\leq\left(2-\frac{2\lambda-2}{2^{\nicefrac{1}{2}}\lambda}\right)\Pr\left(1\right).
\]
 or 
\begin{multline}
\frac{\lambda^{2}-4\lambda+2}{2^{\nicefrac{1}{2}}\lambda}+2\lambda\exp\left(\textrm{-}\lambda\right)\\
\leq\left(2-\frac{2\lambda-2}{2^{\nicefrac{1}{2}}\lambda}\right)\frac{\beta_{0}2^{\nicefrac{1}{2}}\lambda-\lambda^{2}+4\lambda-2}{2\lambda-2}.\label{Ineqhalv1}
\end{multline}
This inequality is checked numerically and the validity is illustrated
in Figure \ref{fig:halv1}.

\begin{figure}
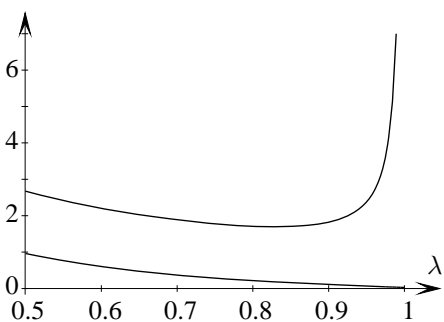

\begin{centering}
\begin{pgfpicture}{28.96mm}{18.91mm}{94.88mm}{61.90mm} 
\pgfsetxvec{\pgfpoint{1.00mm}{0mm}} 
\pgfsetyvec{\pgfpoint{0mm}{1.00mm}} 
\color[rgb]{0,0,0}\pgfsetlinewidth{0.30mm}\pgfsetdash{}{0mm} 
\color[rgb]{1,1,1}\pgfmoveto{\pgfxy(30.96,56.07)}\pgflineto{\pgfxy(33.81,56.07)}\pgflineto{\pgfxy(33.81,57.96)}\pgflineto{\pgfxy(30.96,57.96)}\pgfclosepath\pgffill \color[rgb]{0,0,0}
\pgfputat{\pgfxy(30.96,57.91)}{\pgfbox[bottom,left]{\fontsize{0.08}{0.09}\selectfont }} 
\pgfsetlinewidth{0.12mm}\pgfmoveto{\pgfxy(35.63,22.42)}\pgflineto{\pgfxy(35.63,59.90)}\pgfstroke
\pgfmoveto{\pgfxy(35.63,59.90)}\pgflineto{\pgfxy(34.93,57.10)}\pgflineto{\pgfxy(35.63,58.50)}\pgflineto{\pgfxy(36.33,57.10)}\pgflineto{\pgfxy(35.63,59.90)}\pgfclosepath\pgffill \pgfmoveto{\pgfxy(35.63,59.90)}\pgflineto{\pgfxy(34.93,57.10)}\pgflineto{\pgfxy(35.63,58.50)}\pgflineto{\pgfxy(36.33,57.10)}\pgflineto{\pgfxy(35.63,59.90)}\pgfclosepath\pgfstroke 
\pgfmoveto{\pgfxy(35.59,23.23)}\pgflineto{\pgfxy(90.75,23.23)}\pgfstroke 
\pgfmoveto{\pgfxy(90.75,23.23)}\pgflineto{\pgfxy(87.95,23.93)}\pgflineto{\pgfxy(89.35,23.23)}\pgflineto{\pgfxy(87.95,22.53)}\pgflineto{\pgfxy(90.75,23.23)}\pgfclosepath\pgffill \pgfmoveto{\pgfxy(90.75,23.23)}\pgflineto{\pgfxy(87.95,23.93)}\pgflineto{\pgfxy(89.35,23.23)}\pgflineto{\pgfxy(87.95,22.53)}\pgflineto{\pgfxy(90.75,23.23)}\pgfclosepath\pgfstroke 
\pgfputat{\pgfxy(85.76,19.20)}{\pgfbox[bottom,left]{1}} 
\pgfmoveto{\pgfxy(86.05,23.67)}\pgflineto{\pgfxy(86.05,22.79)}\pgfstroke 
\pgfputat{\pgfxy(74.17,19.20)}{\pgfbox[bottom,left]{0.9}} 
\pgfmoveto{\pgfxy(75.98,23.67)}\pgflineto{\pgfxy(75.98,22.79)}\pgfstroke
\pgfputat{\pgfxy(64.09,19.20)}{\pgfbox[bottom,left]{0.8}} 
\pgfmoveto{\pgfxy(65.90,23.67)}\pgflineto{\pgfxy(65.90,22.79)}\pgfstroke 
\pgfputat{\pgfxy(54.10,19.20)}{\pgfbox[bottom,left]{0.7}} 
\pgfmoveto{\pgfxy(55.84,23.67)}\pgflineto{\pgfxy(55.84,22.79)}\pgfstroke 
\pgfputat{\pgfxy(44.10,19.20)}{\pgfbox[bottom,left]{0.6}} 
\pgfmoveto{\pgfxy(45.76,23.67)}\pgflineto{\pgfxy(45.76,22.79)}\pgfstroke 
\pgfputat{\pgfxy(33.81,19.20)}{\pgfbox[bottom,left]{0.5}} 
%\pgfputat{\pgfxy(34.16,56.57)}{\pgfbox[bottom,left]{\fontsize{3.76}{4.51}\selectfont 7}} 
\pgfmoveto{\pgfxy(36.03,57.13)}\pgflineto{\pgfxy(35.15,57.13)}\pgfstroke 
\pgfputat{\pgfxy(33.00,51.43)}{\pgfbox[bottom,left]{6}} 
\pgfmoveto{\pgfxy(36.03,52.29)}\pgflineto{\pgfxy(35.15,52.29)}\pgfstroke 
%\pgfputat{\pgfxy(34.16,46.88)}{\pgfbox[bottom,left]{\fontsize{3.76}{4.51}\selectfont 5}} 
\pgfmoveto{\pgfxy(36.03,47.45)}\pgflineto{\pgfxy(35.15,47.45)}\pgfstroke 
\pgfputat{\pgfxy(33.00,41.75)}{\pgfbox[bottom,left]{4}}
 \pgfmoveto{\pgfxy(36.03,42.61)}\pgflineto{\pgfxy(35.15,42.61)}\pgfstroke 
%\pgfputat{\pgfxy(34.16,37.20)}{\pgfbox[bottom,left]{\fontsize{3.76}{4.51}\selectfont 3}} 
\pgfmoveto{\pgfxy(36.03,37.76)}\pgflineto{\pgfxy(35.15,37.76)}\pgfstroke 
\pgfputat{\pgfxy(33.00,32.05)}{\pgfbox[bottom,left]{2}} 
\pgfmoveto{\pgfxy(36.03,32.92)}\pgflineto{\pgfxy(35.15,32.92)}\pgfstroke 
%\pgfputat{\pgfxy(34.16,27.52)}{\pgfbox[bottom,left]{\fontsize{3.76}{4.51}\selectfont 1}}
 \pgfmoveto{\pgfxy(36.03,28.08)}\pgflineto{\pgfxy(35.15,28.08)}\pgfstroke 
\pgfputat{\pgfxy(33.00,22.52)}{\pgfbox[bottom,left]{0}} 
\pgfmoveto{\pgfxy(36.03,23.23)}\pgflineto{\pgfxy(35.15,23.23)}\pgfstroke 
\pgfsetlinewidth{0.18mm}\pgfmoveto{\pgfxy(85.54,23.39)}\pgflineto{\pgfxy(86.05,23.37)}\pgfstroke 
\pgfmoveto{\pgfxy(85.03,23.41)}\pgflineto{\pgfxy(85.54,23.39)}\pgfstroke \pgfmoveto{\pgfxy(84.52,23.43)}\pgflineto{\pgfxy(85.03,23.41)}\pgfstroke 
\pgfmoveto{\pgfxy(84.01,23.44)}\pgflineto{\pgfxy(84.52,23.43)}\pgfstroke \pgfmoveto{\pgfxy(83.50,23.46)}\pgflineto{\pgfxy(84.01,23.44)}\pgfstroke 
\pgfmoveto{\pgfxy(35.69,27.89)}\pgfcurveto{\pgfxy(51.38,24.64)}{\pgfxy(67.54,24.14)}{\pgfxy(83.50,23.46)}\pgfstroke \pgfmoveto{\pgfxy(35.65,36.19)}\pgfcurveto{\pgfxy(41.21,34.70)}{\pgfxy(46.87,33.57)}{\pgfxy(52.58,32.81)}\pgfcurveto{\pgfxy(59.13,31.94)}{\pgfxy(65.88,31.12)}{\pgfxy(72.50,31.59)}\pgfcurveto{\pgfxy(76.15,31.84)}{\pgfxy(80.06,32.62)}{\pgfxy(81.97,36.11)}\pgfcurveto{\pgfxy(83.16,38.28)}{\pgfxy(83.62,40.90)}{\pgfxy(83.96,43.35)}\pgfcurveto{\pgfxy(84.09,44.53)}{\pgfxy(84.34,46.89)}{\pgfxy(84.46,48.07)}\pgfcurveto{\pgfxy(84.59,50.34)}{\pgfxy(84.83,54.87)}{\pgfxy(84.96,57.13)}\pgfstroke \pgfputat{\pgfxy(89.06,25.21)}{\pgfbox[bottom,left]{$\lambda$}} 
\end{pgfpicture} 
\par\end{centering}

\protect\caption{\label{fig:halv1}Left and right side for $\lambda$ between $\nicefrac{1}{2}$
and 1.}

\end{figure}
\end{proof}

%\bibliographystyle{ieeetr}
%\bibliography{C:/Users/Peth/Documents/bibtex/database1}

\end{document}